\documentclass[leqno,11pt]{amsart}
\usepackage{amssymb,amsmath,amsfonts}    
\usepackage{amsthm}
\usepackage{color}
\usepackage{hyperref}
\usepackage{verbatim}

\theoremstyle{plain}
\newtheorem{theorem}{Theorem}
\newtheorem{lemma}[theorem]{Lemma}
\newtheorem{corollary}[theorem]{Corollary}
\newtheorem{proposition}[theorem]{Proposition}

\newtheorem{question}[theorem]{Question}

\theoremstyle{definition}

\newtheorem{example}[theorem]{Example}
\newtheorem{remark}[theorem]{Remark}

\newcommand{\eps}{\epsilon}

\newcommand{\ovm}{\overline{m}}

\newcommand{\bI}{{\mathbf I}}
\newcommand{\bJ}{{\mathbf J}}
\newcommand{\bD}{{\mathbf S}}
\newcommand{\bQ}{{\mathbf Q}}
\newcommand{\bL}{{\mathbf \Lambda}}
\newcommand{\bnabla}{{\mathbf \nabla}}

\newcommand{\ba}{{\mathbf a}}
\newcommand{\bb}{{\mathbf b}}
\newcommand{\bc}{{\mathbf c}}

\newcommand{\Heis}{\mathbb H}
\newcommand{\R}{\mathbb R}

\newcommand{\B}{\mathbb B}
\newcommand{\Sph}{\mathbb S}

\newcommand{\cA}{\mathcal A}

\newcommand{\cH}{\mathcal H}
\newcommand{\cI}{\mathcal I}
\newcommand{\cJ}{\mathcal J}
\newcommand{\cK}{\mathcal K}

\newcommand{\cM}{\mathcal M}
\newcommand{\cN}{\mathcal N}
\newcommand{\cP}{\mathcal P}
\newcommand{\cS}{\mathcal S}
\newcommand{\restrict}{\begin{picture}(12,12)
                        \put(2,0){\line(1,0){8}}
                        \put(2,0){\line(0,1){8}}
                       \end{picture}}

\DeclareMathOperator{\length}{length}
\DeclareMathOperator{\diam}{diam}
\DeclareMathOperator{\dist}{dist}

\DeclareMathOperator{\Vol}{Vol}
\DeclareMathOperator{\spa}{span}
\DeclareMathOperator{\spt}{spt}
\DeclareMathOperator{\rad}{rad}

\def\Xint#1{\mathchoice
{\XXint\displaystyle\textstyle{#1}}%
{\XXint\textstyle\scriptstyle{#1}}%
{\XXint\scriptstyle\scriptscriptstyle{#1}}%
{\XXint\scriptscriptstyle\scriptscriptstyle{#1}}%
\!\int}
\def\XXint#1#2#3{{\setbox0=\hbox{$#1{#2#3}{\int}$ }
\vcenter{\hbox{$#2#3$ }}\kern-.6\wd0}}

\def\dashint{\Xint-}
\newcommand{\pa}[1]{\left( #1 \right)}               

\numberwithin{theorem}{section}
\numberwithin{equation}{section}

\title{On uniform measures in the Heisenberg group}
\author{Vasilis Chousionis}
\address{Department of Mathematics \\ University of Connecticut \\ 196 Auditorium Rd U-3009 \\ Storrs, Connecticut 06269-3009}
\email{vasileios.chousionis@uconn.edu}
\author{Valentino Magnani}
\address{Dipartimento di Matematica, Universit\`a di Pisa, Largo Bruno Pontecorvo 5, 56127, Pisa, Italy}
\email{magnani@dm.unipi.it}
\author{Jeremy T. Tyson}
\address{Department of Mathematics \\ University of Illinois \\ 1409 West Green St. \\ Urbana, IL, 61801}
\email{tyson@illinois.edu}
\date{\today}
\thanks{V.C. was funded by the Academy of Finland Grant SA 267047 and Simons Collaboration Grant 521845. J.T.T. supported by NSF grants DMS-1201875 and DMS-1600650.
V.M. supported by the University of Pisa, Project PRA 2018 49.}

\begin{document}

\maketitle

\begin{abstract}
We initiate a classification of uniform measures in the first Heisenberg group $\Heis$ equipped with the Kor\'anyi metric $d_H$, that represents the
first example of a noncommutative stratified group equipped with a homogeneous distance. We prove that $1$-uniform measures are proportional to the spherical $1$-Hausdorff measure restricted to an affine horizontal line, while $2$-uniform measures are proportional to spherical $2$-Hausdorff measure restricted to an affine vertical line. It remains an open question whether $3$-uniform measures are proportional to the restriction of spherical $3$-Hausdorff measure to an affine vertical plane. We establish this conclusion in case the support of the measure is a vertically ruled surface. Along the way, we derive asymptotic formulas for the measures of small extrinsic balls in $(\Heis,d_H)$ intersected with smooth submanifolds. The coefficients in our power series expansions involve intrinsic notions of curvature associated to smooth curves and surfaces in $\Heis$.
\end{abstract}

\section{Introduction}

Let $(X,d)$ be a metric space. A Borel measure $\nu$ on $X$ is said to be {\it $s$-uniform}, where $s \ge 0$, if there exists a constant $c$ so that
\begin{equation}\label{eq:uniform-measure-condition}
\nu(B(x,r)) = cr^s
\end{equation}
for all $x \in X$ and all $r>0$.

Uniform measures feature prominently in the proof of Preiss' celebrated Density Theorem \cite{pre:big} since they occur as tangent measures at points where the $s$-density exists. The classification of uniform measures in Euclidean space is a challenging problem which remains largely open. Marstrand's theorem \cite{mar:density} guarantees that $s$-uniform measures can exist in $\R^n$ only for integer values of $s$. Flat measures are the canonical examples of uniform measures. Recall that a measure $\nu$ in $\R^n$ is said to be {\it $m$-flat}, where $m$ is an integer between $1$ and $n$, if $\nu = c \, \cH^m \restrict V$ for some $m$-dimensional affine subspace $V$ of $\R^n$ and some $c>0$. All $1$- and $2$-uniform measures are flat. However---and this fact is largely responsible for the complexity in the proof of Preiss' Density Theorem---there are non-flat uniform measures. The first example is due to Kowalski and Preiss \cite{kp:besicovitch}, who showed that the measure $\cH^3 \restrict \Sigma$ is a $3$-flat measure in $\R^4$, where
$$
\Sigma = \{ (x_1,x_2,x_3,x_4) \in \R^4 \, : \, x_1^2 = x_2^2 + x_3^2 + x_4^2 \}
$$
is the light cone. In fact, Kowalski and Preiss have completely classified $(n-1)$-uniform measures in $\R^n$ for all values of $n$: every such measure is either $(n-1)$-flat or is proportional to $\cH^{n-1} \restrict M$ where $M$ is an $(n-1)$-dimensional algebraic variety in $\R^n$ which is isometric to $\Sigma \times \R^{n-4}$. A complete classification of $m$-uniform measures in $\R^n$ when $3 \le m < n-1$ remains unknown, although recent work of Nimer \cite{nim1}, \cite{nim2}, \cite{nim3} has improved our understanding and provided new examples. Kirchheim and Preiss \cite{kp:uniformly-distributed} proved that the support of a uniform measure in $\R^n$ is an analytic variety, and Tolsa \cite{tol:uniform} showed that such supports satisfy the David--Semmes `weak constant density' condition and hence are uniformly rectifiable.

In this paper, motivated by our desire to understand the relationship between densities and rectifiability in sub-Riemannian settings, we initiate a study of uniform measures in stratified groups. As a point of departure we consider the first Heisenberg group $\Heis$. We recall that $\Heis$ can be identified with $\R^3$, equipped with the nonabelian group law
\begin{equation}\label{eq:groupHeis}
(x_1,x_2,x_3)*(y_1,y_2,y_3) = (x_1+y_1,x_2+y_2,x_3+y_3-2x_1y_2+2x_2y_1).
\end{equation}
We will denote points of $\Heis$ with the notation $x=(x_1,x_2,x_3)$. We denote by $\pi:\Heis\to\R^2$ the projection $\pi(x_1,x_2,x_3)=(x_1,x_2)$. We will exploit the special symmetries of $\Heis$ via the rotations $\tilde R:\Heis\to\Heis$ defined by the matrix
\begin{equation}\label{eq:Rotations}
\left(\begin{array}{ccc}
R_1 & R_2 & 0 \\
0 & 0 & 1
\end{array}\right),
\end{equation}
where $R_1,R_2\in\R^2$ are orthogonal vectors and $-R_2=\bJ R_1$, with $\bJ$ as in \eqref{eq:IJ}.

Since the notion of uniform measure depends strongly on the choice of metric, we fix a specific metric on the Heisenberg group. We work throughout this paper with the Kor\'anyi metric
\begin{equation}\label{eq:d_H}
d_H(x,y) = ||x^{-1}*y||_H, \qquad ||(x_1,x_2,x_3)||_H^4 := (x_1^2+x_2^2)^2+x_3^2.
\end{equation}
One can verify that each matrix $\tilde R$ as in \eqref{eq:Rotations} is an isometry with respect to the distance $d_H$ of \eqref{eq:d_H}
and is also a Lie group homomorphism with respect to the group operation \eqref{eq:groupHeis}.

Our choice of this metric stems from three facts. First, Marstrand's density theorem holds for the metric space $(\Heis,d_H)$ \cite{ct:marstrand}; this fact ensures that supports of uniform measures in $(\Heis,d_H)$ are highly regular. Second, the rotational symmetry of the Kor\'anyi metric about the $x_3$-axis simplifies the area formula for submanifolds
in the Heisenberg group \cite{mv:intrinsic,mag:blowup2step,mag:blowup_heis}.
Finally, the explicit nature of this metric allows us to compute explicitly the coefficients of terms arising in local power series expansions of the volumes of small extrinsic balls along submanifolds. Such computations are crucial for us to obtain explicit differential equations governing the supports of uniform measures.

We now introduce the main results of the paper. Let $\mu$ be an $s$-uniform measure on $(\Heis,d_H)$. It follows from the main results of \cite{ct:marstrand} that $s=\overline{m}$ is an integer and the support of $\mu$ is a real analytic variety whose top dimensional stratum has topological dimension $m$ and Hausdorff dimension $\overline{m}$. (See \eqref{eq:stratification} for more information about the stratification of analytic varieties.) Here the pair $(m,\overline{m})$ can be any of the following:
\begin{equation}\label{list}
(0,0), \quad (1,1), \quad (1,2), \quad (2,3), \quad (3,4).
\end{equation}
Observe that these various cases are distinguished by the value of the sub-Riemannian Hausdorff dimension $\ovm$. The cases $\overline{m}=0$ and $\overline{m}=4$ being trivial, we restrict our attention to the cases $\ovm\in\{1,2,3\}$. In view of the results of Kowalski and Preiss, it is natural to ask the following question.

\begin{question}\label{conj}
Let $\mu$ be an $\ovm$-uniform measure on $\Heis$. Must $\mu$ be proportional to $\cS^{\ovm}\restrict G$, where $G$ is a left coset of a homogeneous subgroup of $\Heis$?
\end{question}

A homogeneous subgroup of $\Heis$ is a Lie subgroup which is invariant under the action of the dilation semigroup. We will use the modifier {\it affine} to refer to a left coset of such a subgroup. Question \ref{conj} asks whether the supports of uniform measures in $\Heis$ are {\it flat}, i.e., no exotic examples such as the Kowalski--Preiss light cone exist. Observe that there do exist uniform measures with non-flat support in Heisenberg groups of sufficiently high dimension, see \cite[Proposition 4.1]{ct:marstrand}.

It is easy to observe that any homogeneous subgroup of the Heisenberg group either contains the vertical direction or does not.
In the former case it is called a vertical subgroup and in the latter case it is called a horizontal subgroup, see for instance \cite{mss:intrinsic,Mag14} or \cite{ct:marstrand} for details. In view of this classification, Question \ref{conj} can be restated as follows.

\begin{question}\label{conj-revised}
Let $\mu$ be a $\ovm$-uniform measure on $\Heis$.
\begin{itemize}
\item[(1)] If $\ovm = 1$, must $\mu$ be proportional to $\cS^1\restrict L$ for some affine horizontal line $L$?
\item[(2)] If $\ovm = 2$, must $\mu$ be proportional to $\cS^2\restrict V$ for some affine vertical line $V$?
\item[(3)] If $\ovm = 3$, must $\mu$ be proportional to $\cS^3\restrict W$ for some affine vertical plane $W$?
\end{itemize}
\end{question}

Our first two results answer parts (1) and (2) of this question in the affirmative.

\begin{theorem}\label{thmA}
Let $\mu$ be a $1$-uniform measure on $\Heis$. Then $\mu$ is proportional to $\cS^1 \restrict L$ where $L$ is an affine horizontal line.
\end{theorem}

\begin{theorem}\label{thmB}
Let $\mu$ be a $2$-uniform measure on $\Heis$. Then $\mu$ is proportional to $\cS^2\restrict V$, where $V$ is an affine vertical line.
\end{theorem}

Part (3) of Question \ref{conj-revised} remains open. However, we are able to answer the question in a special case.

\begin{theorem}\label{thmC}
Let $\mu$ be a $3$-uniform measure on $\Heis$. Then $\spt\mu$ is a two-dimensional real analytic variety. If $\spt\mu$ is a vertically ruled real analytic surface, then $\mu$ is proportional to $\cS^3\restrict W$, where $W$ is an affine vertical plane.
\end{theorem}

Following the method developed by Kowalski and Preiss, we derive Theorems \ref{thmA}, \ref{thmB} and \ref{thmC} from local power series expressions for the spherical Hausdorff measure of an extrinsic ball of $(\Heis,d_H)$ intersected with a smooth submanifold. Such formulas, which appear in Propositions \ref{propA}, \ref{propB} and Theorem~\ref{theoC}, are of independent interest. In \cite{clt:uniform-rectifiability}, the first and the second author use \eqref{eqnC} to study potential definitions of uniform rectifiability in the Heisenberg group. These local power series expressions would also be useful in studying the more general problem of characterizing {\it locally uniform measures}, i.e., measures for which the uniformity condition \eqref{eq:uniform-measure-condition} holds at each point of the support for balls of sufficiently small radii.

In the proof of Theorem \ref{thmC} we use an important tool of independent interest,
that is the following geometric PDE
\begin{equation}\label{PDE-C}
\cH_0^2 + \frac32 \cP_0^2 + 4 \vec{e}_1(\cP_0) = 0,
\end{equation}
that is satisfied at all noncharacteristic points of any real analytic surface without boundary contained in the support of a $3$-uniform measure in $\Heis$.
Here $\cH_0(x)$ denotes the horizontal mean curvature, $\cP_0(x)$ denotes the Arcozzi--Ferrari imaginary curvature (curvature of the metric normal), and $(\vec{e}_1)_x$ denotes the characteristic vector at a noncharacteristic point $x$ in a surface $\Sigma$. We refer to \eqref{H0-level-set}, \eqref{P0-level-set} and \eqref{e1-level-set} for definitions of these quantities. 

In Euclidean space the support of a uniform measure can be given by a quadratic surface, as the light cone in $\R^4$ that has been previously mentioned.
It is then reasonable to establish whether analogous cases may occur in $\Heis$.
An application of \eqref{PDE-C} answers this question in the case
of {\it quadratic $x_3$-graphs}, namely 2-dimensional graphs of a quadratic form in the variables $x_1$ and $x_2$.
The fact that the $x_3$ variable has degree two
shows that the family of these quadratic surfaces is invariant by 
the dilations of the Heisenberg group.

\begin{proposition}\label{prop:no-quadratic-t-graphs-as-3-uniform-supports}
No quadratic $x_3$-graph can be the support of a $3$-uniform measure.
\end{proposition}

To conclude this introduction, we outline the proofs of our main theorems (Theorem \ref{thmA}, \ref{thmB} and \ref{thmC}). Our sketch of the proof of Theorem \ref{thmC} will illustrate the role of \eqref{PDE-C}.

To prove Theorem \ref{thmA}, we write the support $M$ of $\mu$ as a stratified analytic $1$-variety:
$$
M = M_{(0)} \cup M_{(1)}.
$$
The area formula \eqref{eq:area-formula} joined with Theorem~\ref{zeroth2} immediately show that each curve contained in the $1$-dimensional stratum $M_{(1)}$ is fully horizontal. We prove that every such curve has vanishing horizontal curvature, and hence is a horizontal line. It remains to show that there is no $0$-dimensional stratum. We rule out endpoints and junction points of degree at least three by volume considerations, and we rule out corners (junction points of degree two) by a geometric argument. Eventually, we conclude that $M_{(0)} = \emptyset$ and $M_{(1)} = M = \spt\mu$ is a single affine horizontal line.

To prove Theorem \ref{thmB}, we again write the support $M$ of $\mu$ as a stratified analytic $1$-variety:
$$
M = M_{(0)} \cup M_{(1)}.
$$
By a blow-up argument, similar to that of \cite{mag:nnote}, we show that each curve contained in the $1$-dimensional stratum $M_{(1)}$ is fully nonhorizontal. If $\Sigma:I \to \Heis$ is such a curve, and $\gamma$ is the projection of $\Sigma$ to $\R^2$, then either $\Sigma$ is contained within a vertical line, or $\sigma:=|\tfrac{d}{ds}\gamma| > 0$ on a dense open set $J \subset I$. In the latter case we derive several differential equations satisfied by the speed and curvature of $\gamma$. Specifically, we show that the speed $\sigma$ and curvature of $\gamma$ are proportional, and that $\sigma$ satisfies several additional nonlinear ordinary differential equations. We show that there are no nonzero solutions to this ODE system. Hence the second case cannot occur, and so $\Sigma$ must be entirely contained within a vertical line. Since the support of $\mu$ must be connected, we conclude that $M_{(0)} = \emptyset$ and $M_{(1)} = M = \spt \mu$ is a single affine vertical line.

Finally, we discuss the proof of Theorem \ref{thmC}. Our assumption is that the support $M$ of $\mu$ is a real analytic, noncharacteristic, vertically ruled surface. In this situation the geometric PDE \eqref{PDE-C} reduces to $\cH_0 = 0$. By the Bernstein theorem for the Heisenberg group, $M$ must be contained in a vertical plane. Volume considerations ensure that there are no boundary edges, hence $M$ is a single affine vertical plane.

\

\paragraph{\bf Acknowledgements.} Research for this paper was initiated during a `Research in Pairs' event at the Centro Internazionale per la Ricerca Matematica in Trento, Italy in June 2014. The authors wish to gratefully acknowledge the hospitality of CIRM and the University of Trento during this period. The first and third authors would also like to thank the Department of Mathematics at the University of Bern for its hospitality during June 2015, when additional research for the paper was conducted. We would also like to thank Andrea Merlo for valuable comments.

\section{Background}\label{sec:background}

\subsection{Hausdorff measures and spherical Hausdorff measures}\label{subsec:Hausdorff}

We recall that the Hausdorff $s$-measure of a set $A$ in a metric space $(X,d)$ is defined to be
$$
\cH^s(A) = \lim_{\delta \to 0} \inf \left\{ \sum_{i=1}^\infty (\diam E_i)^s \, : \, A \subset \bigcup_{i=1}^\infty E_i, \quad \diam E_i < \delta \right\},
$$
while the spherical Hausdorff measure $\cS^s(A)$ of $A$ is defined by the same formula, but with the additional assumption that the sets $E_i$ are restricted to be balls $B(x_i,r_i)$ of $(X,d)$. It is well known that $\cH^s$ and $\cS^s$ are comparable on any metric space:
\begin{equation}\label{H-S-comparable}
\cH^s(A) \le \cS^s(A) \le 2^s \cH^s(A) \qquad \forall \, A \subset X,
\end{equation}
and that there exist sets (even subsets of Euclidean space) for which these two measures differ. In the Euclidean setting, $\cH^m$ and $\cS^m$ coincide for $m$-rectifiable sets. A corresponding result is not yet known in the sub-Riemannian setting. See \cite[Remark 4.41]{sc:GMTtopics} for the latest information on this question. For the purposes of the area formula on submanifolds discussed in the following subsection, we focus attention to the spherical Hausdorff measure. Note that \eqref{H-S-comparable} ensures that the notion of Hausdorff dimension is the same when defined by either the class of Hausdorff measures or spherical Hausdorff measures:
$$
\dim A = \inf \{ s \ge 0 : \cH^s(A) = 0 \} = \inf \{ s \ge 0 : \cS^s(A) = 0 \}.
$$
We remark that we denote by $\dim A$ the Hausdorff dimension of a set $A$ in any metric space. Our primary interest will be in subsets of the first Heisenberg group $\Heis$ endowed with the Kor\'anyi metric.

\subsection{Sub-Riemannian differential geometry of curves and surfaces in $\Heis$}\label{subsec:curvatures}

We consider the left-invariant vector fields $X_1,X_2,X_3$ in $\Heis$, which agree with
the standard unit vectors $e_1,e_2,e_3$ at $o=(0,0,0)$. Denoting $x=(x_1,x_2,x_3)$ we have
$$
X_1(x) = e_1 + 2x_2e_3, \quad X_2(x) = e_2 - 2x_1e_3 \quad \mbox{and}
\quad X_3(x) = e_3.
$$
As first-order differential operators, these vector fields act as follows:
$$
(X_1)_x = \partial_1 + 2x_2\partial_3, \quad (X_2)_x = \partial_2 -
2x_1\partial_3 \quad \mbox{and} \quad (X_3)_x = \partial_3.
$$
We observe that $[X_1,X_2]=-4X_3$, thus the horizontal distribution $H\Heis$, whose fiber at $x\in\Heis$ is $H_x\Heis=\spa\{X_1(x),X_2(x)\}$, is completely nonintegrable. We also refer to $H_x\Heis$ as the {\it horizontal tangent space} at $x$. The horizontal distribution can also be defined fiberwise as the kernel of the defining contact $1$-form
\begin{equation}\label{contact-form}
\vartheta = dx_3 + 2x_1 \, dx_2 - 2x_2 \, dx_1 \, .
\end{equation}
The length of a horizontal tangent vector $v \in H_x\Heis$ is defined to be $|v|_x := \sqrt{c_1^2+c_2^2}$, where $v = c_1 (X_1)_x + c_2 (X_2)_x$. For a function $f:\Heis\to\R$, we denote by $\nabla_0f = X_1fX_1+X_2fX_2$ the horizontal gradient of $f$ and by $|\nabla_0f| = \sqrt{(X_1f)^2+(X_2f)^2}$ its norm.

We denote by $\cJ:H\Heis \to H\Heis$ the fiberwise linear operator defined in the horizontal distribution by the rule $\cJ(X_1) = X_2$ and $\cJ(X_2) = -X_1$.

\

\paragraph{\bf Curves.} A point $x = \Sigma(s)$ on a $C^1$ curve $\Sigma$ in $\Heis$ is {\it horizontal} if $\vartheta(\dot\Sigma(s)) = 0$, i.e., if $\dot\Sigma(s)$ is a horizontal tangent vector at $x$. We say that $\Sigma$ is {\it horizontal} if every point on $\Sigma$ is horizontal. The {\it Carnot--Carath\'eodory length} of a horizontal curve $\Sigma:I \to \Heis$ is defined to be
$$
\length_{cc}(\Sigma) = \int_I |\dot\Sigma(s)|_{\Sigma(s)} \, ds,
$$
where $|v|_x$ denotes the horizontal length of a horizontal tangent vector $v$ at $x \in \Heis$. Every horizontal curve can be reparameterized by arc length, in which case $|\dot\Sigma(s)|_{\Sigma(s)} = 1$ for all $s$. Observe that if $\Sigma$ is a horizontal curve in $\Heis$ parameterized by arc length, then $\gamma = \pi \circ \Sigma$ is a curve in $\R^2$ parameterized by arc length: $|\dot\gamma(s)| = 1$ for all $s$.

A $C^1$ curve $\Sigma$ in $\Heis$ is {\it fully nonhorizontal} if $\vartheta(\dot\Sigma(s)) \ne 0$ for all $s$. A fully nonhorizontal curve $\Sigma$ is said to be {\it parameterized by homogeneous arclength} if
$$
\vartheta(\dot{\Sigma}(s)) = 1
$$
for all $s$. This terminology is justified by the area-type formula
\[
\cS^2(\Sigma)=\int_I |\vartheta(\dot\Sigma(s))| ds,
\]
that is a special case of \eqref{eq:area-formula}. The formula for $\cS^2(\Sigma)$  also shows that any fully nonhorizontal curve can always be reparameterized by homogeneous arclength.

\

\paragraph{\bf Surfaces.} A point $x$ on a smooth surface $\Sigma$ in $\Heis$ is said to be {\it noncharacteristic} if $T_x\Sigma \ne H_x\Heis$. The set of all characteristic points of $\Sigma$ is the {\it characteristic set}, denoted $C(\Sigma)$. A surface is {\it fully noncharacteristic} if $C(\Sigma) = \emptyset$.

Let $x$ be a noncharacteristic point on a smooth surface $\Sigma$. A {\it horizontal normal} to $\Sigma$ at $x$ is the projection of a Riemannian normal into the horizontal plane at $x$. If $\Sigma$ is given (locally at $x$) as a level set of a smooth function $u$, then
\begin{equation}\label{n0-level-set}
\vec{n}_0(x) = \nabla_0u(x)/|\nabla_0u(x)| = |\nabla_0u(x)|^{-1}(X_1uX_1 + X_2uX_2)
\end{equation}
is a choice of horizontal normal. Relative to a choice of horizontal normal,
we introduce the {\bf characteristic vector field} $\vec{e}_1 = \cJ(\vec{n}_0)$, defined at noncharacteristic points of a surface $\Sigma$. In the level set formulation,
\begin{equation}\label{e1-level-set}
\vec{e}_1(x) = |\nabla_0u(x)|^{-1}(X_1uX_2 - X_2uX_1)\,.
\end{equation}
There exists a unique unit-speed horizontal curve $c_x:(-\eps,\eps) \to \Sigma$, passing through $x$, whose velocity vector $\dot{c_x}(0) \in H_x\Heis\cap T_x\Sigma =: HT_x\Sigma$ agrees with the characteristic vector field $\vec{e}_1$ at $x$. The noncharacteristic portion $\Sigma \setminus C(\Sigma)$ is foliated by such curves; this is the so-called {\it Legendrian foliation}.

The {\it horizontal mean curvature} of $\Sigma$ at $x$, denoted $\cH_0(x)$, is the planar curvature of the projection of the Legendrian curve $c_x$ into the first two coordinates. In other words,
$$
\cH_0(x) := k_{\pi\circ c_x}(0).
$$
where $k_\gamma$ denotes the planar curvature of $\gamma$. We also call $\cH_0(x)$ the {\it horizontal curvature} of the space curve $c_x$ at the point $x$.
In the case when $\Sigma$ is a level set of a function $u$,
\begin{equation}\label{H0-level-set}
\cH_0 = X_1 \left( \frac{X_1u}{|\nabla_0u|} \right) +
X_2 \left( \frac{X_2u}{|\nabla_0u|} \right).
\end{equation}
See \cite[\S4.3]{cdpt:survey} for further information.

Next we recall the concepts of metric normal and imaginary curvature introduced by Arcozzi and Ferrari \cite{af1}, \cite{af2}. The {\it metric normal} at a noncharacteristic point $x \in \Sigma$, denoted $\cN_x\Sigma$, is the set of points $y \in \Heis$ such that $\dist_{cc}(y,\Sigma)=d_{cc}(y,x)$. Here $d_{cc}$ denotes the {\it Carnot-Carath\'eodory (C--C) metric}. According to \cite[Theorem 1.2]{af1}, $\cN_x\Sigma$ consists of a nontrivial arc contained in a CC geodesic passing through $x$, whose tangent vector at $x$ is the horizontal unit normal $\vec{n}_0$ to $\Sigma$. The {\it imaginary curvature} of $\Sigma$ at $x$, denoted $\cP_0(x)$, is the horizontal curvature of $\cN_x\Sigma$ at $x$. In other words, $\cP_0(x)$ is the signed curvature of the planar projection of $\cN_x\Sigma$. In the level set formalism,
\begin{equation}\label{P0-level-set}
\cP_0 = \frac{4X_3u}{|\nabla_0u|}.
\end{equation}
Observe that $\cP_0$ vanishes on an open subset $U$ of $\Sigma$ if and only if $U$ is vertically ruled; in the level set formalism this means that the defining function $u$ depends only on $x_1$ and $x_2$. The terminology `imaginary curvature' comes from \cite{af2}, where the same authors use the horizontal mean
and imaginary curvatures to study the horizontal Hessian of the Carnot--Carath\'eodory distance function $\delta_{\Sigma,cc}(x) = \dist_{cc}(x,\Sigma)$.
According to \cite[Theorem 1.1]{af2}, the horizontal Hessian of $\delta_{\Sigma,cc}$ at $x$ is given by
$$
\vec{n}_0(x) \otimes \vec{n}_0(x) \cdot ( \cH_0(x) \, \bI + \cP_0(x) \, \bJ ),
$$
where
\begin{equation}\label{eq:IJ}
\bI = \begin{pmatrix} 1 & 0 \\ 0 & 1 \end{pmatrix} \quad \mbox{and}
\quad \bJ = \begin{pmatrix} 0 & 1 \\ -1 & 0 \end{pmatrix}
\end{equation}
denote the $2\times 2$ identity and standard symplectic matrices. For this reason the authors of \cite{af2} also term $\cH_0(x)$ the {\it real curvature} of $\Sigma$ at $x$.

Note that the mean and imaginary curvatures, as well as the characteristic vector field, depend up to sign on a choice of normal to the surface, or equivalently, on a choice of the defining function $u$. However, $\cH_0^2$, $\cP_0^2$ and $\vec{e}_1(\cP_0)$ are all independent of such choice. The left hand side of the PDE \eqref{PDE-C} is therefore intrinsic to the surface.

\subsection{Area formula for submanifolds of the Heisenberg group}\label{subsec:area}

Propositions \ref{propA}, \ref{propB} and Theorem~\ref{theoC} derive from an area formula expressing the natural sub-Riemannian volume locally on a smooth submanifold of $(\Heis,d_H)$ in terms of a parameterization.

The topological dimension $m$ and the sub-Riemannian Hausdorff dimension $\ovm$ of a smooth submanifold $\Sigma$ in any stratified group are related by a formula of Gromov, \cite[0.6.B]{gro:cc}. This formula has an algebraic formulation in terms of the degree of points, that also leads to the computation of the appropriate spherical measure; see \cite{mv:intrinsic} and \cite{mag:towards_area}.

In the first Heisenberg group, the degree $d_\Sigma(x)$ a point $x$ of a $C^1$ submanifold $\Sigma$ has the following expression:
\begin{equation}\label{d:degree}
d_\Sigma(x)=\left\{
\begin{array}{ll}
1 & \text{if $\dim\Sigma=1$ and $T_x\Sigma\subset H_x\Heis$} \\
2 & \text{if $\dim\Sigma=1$ and $T_x\Sigma\not\subset H_x\Heis$} \\
2 & \text{if $\dim\Sigma=2$ and $T_x\Sigma=H_x\Heis$} \\
3 & \text{if $\dim\Sigma=2$ and $T_x\Sigma\not\subset H_x\Heis$}
\end{array}\right. .
\end{equation}
The algebraic definition of degree in any stratified group can be found in \cite{mv:intrinsic}. Gromov's formula for the Hausdorff dimension of a submanifold $\Sigma$ is
\[
\max_{x\in\Sigma} d_\Sigma(x).
\]
In view of the Frobenius theorem, any 2-dimensional smooth submanifold of $\Heis$ cannot contain only points of degree two, and hence the previous formulae show that all possible pairs $(m,\ovm)$ are given by the list in \eqref{list}. For instance a smooth curve $\Sigma$ has Hausdorff dimension $\ovm=1$ if and only if $\Sigma$ is everywhere horizontal and every smooth surface $\Sigma$ in $(\Heis,d_H)$ has Hausdorff dimension $\ovm$ equal to $3$.

The next result follows by joining \cite{mv:intrinsic}, \cite[(3)]{mag:blowup2step} and \cite[Proposition~4.5]{mag:blowup_heis}.

\begin{theorem}[Area formula on submanifolds of $(\Heis,d_H)$]\label{th:magnani-area-formula}
Let $\Sigma$ be a $C^{1,1}$ submanifold of topological dimension $m\in\{1,2\}$ and Hausdorff dimension $\ovm\in\{1,2,3\}$ in $(\Heis,d_H)$. Let $\Omega \subset \Sigma$ be a domain, parameterized by a mapping $\Phi:U \subset \R^m \to \Heis$. Then there exists a geometric constant
$\beta_d$, depending on $d_H$ and $(m,\ovm)$ such that
\begin{equation}\label{eq:area-formula}
\cS^{\ovm}(\Omega) = \beta_d \int_U |(\partial_{x_1}\Phi \wedge \partial_{x_2}\Phi \wedge \cdots \wedge \partial_{x_m}\Phi)_{(\ovm)}| \, dx,
\end{equation}
where $(v_1\wedge\cdots\wedge v_m)_{(d)}$ denotes the projection of the $m$-vector $v_1\wedge\cdots\wedge v_m$ into the $m$-vectors of degree $\ovm$.
\end{theorem}

Such area formula holds in much greater generality than what we state here, as for higher dimensional Heisenberg groups and general stratified groups
equipped with a variety of different homogeneous distances. Recently, a general approach
to area formulas has been established in \cite{mag:towards_area}, also proving
these formulas for new classes of submanifolds of class $C^1$.

For simplicity, we henceforth assume that the constant $\beta_d$ in Theorem~\ref{th:magnani-area-formula} equals one. This can be achieved by rescaling the
spherical Hausdorff measure $\cS^{\ovm}$, if necessary. Clearly such rescaling has no effect on the arguments in this paper.

\begin{lemma}\label{zeroth1}
Let $\Sigma$ be a $C^{1,1}$ submanifold of topological dimension $m\in\{1,2\}$ and Hausdorff dimension $\ovm\in\{1,2,3\}$.
Then for all points $x\in\Sigma$ of degree $\ovm$, the following limit
\begin{equation}\label{eq:limS^mr^-m}
\lim_{r\to 0} \frac{\cS^{\ovm}(B_H(x,r) \cap \Sigma)}{r^{\ovm}}
\end{equation}
exists and equals a fixed geometric constant.
\end{lemma}

\begin{proof}
From \eqref{eq:area-formula}, using the rescaled version of $\cS^{\ovm}$ such that $\beta_d=1$, we have
\[
\begin{split}
\frac{\cS^{\ovm}(B_H(x,r) \cap \Sigma)}{r^{\ovm}}&=r^{-\ovm}\int_{\Phi^{-1}(B_H(x,r) \cap \Sigma)}
|(\partial_{x_1}\Phi \wedge \partial_{x_2}\Phi \wedge \cdots \wedge \partial_{x_m}\Phi)_{(\ovm)}| \, dx \\
&=\pa{\dashint_{\Sigma\cap B_H(x,r)} |\big(\tau^{\ovm}_\Sigma(x)\big)| d\tilde\mu_m}\pa{\frac{\tilde\mu_m(\Sigma\cap B_H(x,r))}{r^{\ovm}}  }
\end{split}
\]
where $\tilde\mu_m$ is the $m$-dimensional surface measure on $\Sigma$
with respect to a fixed auxiliary Riemammian metric $\tilde g$,
$\tau_\Sigma$ is the unit $m$-tangent vector with respect to
the metric $\tilde g$ and $\tau^{\ovm}_\Sigma=(\tau_\Sigma)_{(\ovm)}$ is
the homogeneous part of $\tau_\Sigma$ having degree $\ovm$.
As a result, the continuity of $x\to |\tau^{\ovm}_\Sigma(x)|$ and \cite[(1.3)]{mv:intrinsic} show that the
limit \eqref{eq:limS^mr^-m} exists and equals to the $m$-dimensional Euclidean surface measure of
$B_H(0,1)\cap \Pi$ for a suitable subspace $\Pi$ of $\Heis$.

When $m=1$ and $\ovm=1$, $\Pi$ is a one dimensional homogeneous
subgroup of $\Heis$ contained in the $x_1x_2$-plane, so by the rotational symmetry of $B_H(0,1)$ about the $x_3$-axis the
Euclidean surface measure of $B_H(0,1)\cap \Pi$ does not depend on $\Pi$. In this case the value of the limit \eqref{eq:limS^mr^-m} is equal to $2$.

When $m=1$ and $\ovm=2$, $\Pi$ is uniquely determined by the one dimensional vertical subgroup of $\Heis$. Again, the value of the limit \eqref{eq:limS^mr^-m} is equal to $2$.

Finally, when $m=2$ and $\ovm=3$, $\Pi$ is some 2-dimensional homogeneous subgroups of $\Heis$, corresponding to a vertical plane of $\Heis$.
Also in this case, due to the special symmetry of $B_H(0,1)$, the Euclidean surface measure of $B_H(0,1)\cap \Pi$ does not depend
on the choice of this vertical plane $\Pi$. The value of the limit \eqref{eq:limS^mr^-m} in this case is explicitly computed in the Appendix to this paper, see \eqref{eq:omegaH}.

This concludes the proof.
\end{proof}

Let $\Sigma$ be an analytic variety. We denote by
\begin{equation}\label{eq:stratification}
\Sigma = \Sigma_{(0)} \cup \cdots \cup \Sigma_{(m)}
\end{equation}
the stratification of $\Sigma$ into a countable union of analytic submanifolds of dimensions between $0$ and $m$. The existence of such a stratification is a celebrated theorem of Lojasiewicz. We refer the reader to \cite[Section 6.3]{krprks:realanalytic} for further discussion. Apart from trivial cases, the stratification in \eqref{eq:stratification} does not contain submanifolds of full dimension, i.e., $m<n$.

\begin{theorem}\label{zeroth2}
If $\mu$ is an $\ovm$-uniform measure, then
\begin{equation}\label{eq:Sm}
\mu = c \, \cS^{\ovm}\restrict \spt\mu.
\end{equation}
\end{theorem}
\begin{proof}
By \cite[Proposition~3.1]{ct:marstrand}, the support
$\Sigma $ is an analytic variety with stratification \eqref{eq:stratification}.
Let us define $S_0=\Sigma\setminus\Sigma_{(m)}$.
If $m=2$ and $\ovm=3$, then $S_0$ is made by countable
unions of points and analytic curves, hence it is $\cH^2_E$-negligible.
By the absolute continuity of $\cS^3$ with respect to
the 2-dimensional Euclidean Hausdorff measure $\cH^2_E$, see for instance \cite{BTW2009}, we get $\cS^3(S_0)=0$.
If $m=1$ and either $\ovm=2$ or $\ovm=1$, then $S_0$ is contained in an at most countable union of points, therefore $\cS^2(S_0)=0$.
We conclude that in any case $\cS^{\ovm}(S_0)=0$.

The set $\Sigma_{(m)}$ contains a possibly nonempty subset
$\Sigma_{0,m}$ made of points with degree less than $\ovm$,
which is $\cS^{\ovm}$-negligible, due to
\cite[Corollary~1.2]{mag:blowup2step}.
Lemma~\ref{zeroth1} provides us in particular an
asymptotically doubling property of the measure
\[
\phi=\cS^{\ovm}\restrict\Sigma'_{(m)}
\]
with $\Sigma'_{(m)}=\Sigma_{(m)}\setminus\Sigma_{0,m}$
We are then in the position to apply both  \cite[2.8.17]{fed:gmt} and \cite[2.9.7]{fed:gmt} to $\phi$, hence
differentiating the restriction
\[
\mu\restrict\Sigma'_{(m)}
\]
with respect to $\phi$ in $\Sigma_{(m)}'$, we get
the equality \eqref{eq:Sm} on $\Sigma'_{(m)}$.
The uniform condition immediately shows that $\mu$
and $\cS^{\ovm}\restrict\Sigma$ are mutually absolutely continuous,
therefore both $S_0$ and $\Sigma_{0,m}$ are also $\mu$-negligible.
This concludes the proof.
\end{proof}

\section{1-uniform measures and smooth horizontal curves}\label{sec:1-uniform}

In this section we analyze $1$-uniform measures, and in particular, we prove Theorem \ref{thmA}. The proof of Theorem \ref{thmA} relies on the following power series formula for the volumes of small Kor\'anyi balls along smooth horizontal curves.

\begin{proposition}\label{propA}
Let $\Sigma$ be a $C^3$ smooth horizontal curve in $\Heis$, and let $x = \Sigma(0)$. Set $\gamma = \pi \circ \Sigma$ and denote by $k = k_\gamma$ the planar curvature of $\gamma$. Then
\begin{equation}\label{eqnA}
\cS^1(B_H(x,r)\cap\Sigma) = 2r + \ba_1(x) r^3 + o(r^3),
\end{equation}
where $\ba_1(x)$ is a positive multiple of $k(0)^2$.
\end{proposition}

We first prove Theorem \ref{thmA} under the assumption that Proposition \ref{propA} is satisfied.

\begin{proof}[Proof of Theorem \ref{thmA}]
Let $\mu$ be a $1$-uniform measure on $\Heis$. By \cite[Proposition~3.1]{ct:marstrand}, $M = \spt\mu$ is a real analytic variety.
In view of Lojasiewicz's Structure Theorem for real analytic varieties, Theorem~\ref{zeroth2} with $\ovm=1$ and \eqref{eq:area-formula}, we deduce that $M = M_{(0)} \cup M_{(1)}$ is the union of countably many analytic submanifolds whose dimensions are at most one. Moreover, all curves contained in $M_{(1)}$ must be fully horizontal.

Due to the $1$-uniformity of $\mu$, $M_{(0)}$ cannot contain isolated points. Moreover, we cannot have points $x$ of $M_{(0)}$ where more than two curves in $M_{(1)}$ meet, since the measures of small balls centered at $x$ would be too large. The $1$-uniformity of $\mu$ also prevents curves of $M_{(1)}$ from terminating at some point $x$, since the measures of small balls centered at $x$ would be too small.

Analyticity ensures that whenever $x\in M_{(1)}$, there exists $r_x>0$ such that $M \cap B_H(x,r_x)$ is a single connected analytic curve $\Sigma$. Indeed, if we had an infinite sequence of points of $M$ not lying in this curve and converging to $x$, then---since $M$ does not contain isolated points---we would conclude that $x$ lies in the intersection of at least two horizontal curves. This gives a contradiction. We are thus in a position to apply Proposition~\ref{propA} at $x$, concluding that the curvature of $\gamma = \pi \circ \Sigma$ at $\pi(x)$ is equal to zero.

It follows that any horizontal curve in $M_{(1)}$ is an affine horizontal line. We have proved that $M$ is the union of disjoint affine horizontal
line segments (bounded or unbounded), two of these line segments may meet at a corner, but there are no endpoints of line segments or junction points of three or more lines.

\begin{lemma}\label{lem:1-uniform-no-corners}
$M$ cannot contain corners, i.e., $M$ cannot contain two horizontal line segments meeting transversally at a common endpoint.
\end{lemma}

Assuming the lemma, we now know that $M$ consists of one or more bi-infinite affine horizontal lines. However, $M$ cannot contain two such lines: consider a ball centered on one such line with radius slightly larger than the distance to the closest other line in $M$. We conclude that $M$ consists of a single, bi-infinite, affine horizontal line and $\mu$ is proportional to $\cS^1 \restrict M$. This completes the proof of Theorem \ref{thmA}, modulo Lemma \ref{lem:1-uniform-no-corners} and Proposition \ref{propA}.
\end{proof}

\begin{proof}[Proof of Lemma \ref{lem:1-uniform-no-corners}]
Suppose that $M = \spt\mu$ contains two horizontal line segments meeting at a common endpoint $x$. Let us denote by $\gamma_1,\gamma_2$ these two horizontal line segments. We assume that both $\gamma_1$ and $\gamma_2$ are parameterized on half-open intervals, and do not contain $x$. By analyticity, there exists some small ball $B$ centered at $x$ so that $M$ contains no points in $B$ apart from $\gamma_1 \cup \{x\} \cup \gamma_2$. Choose a point $y \in \gamma_1$ with $\dist_H(x,y) < \tfrac1{100} \rad(B)$. By assumption $\mu(B_H(y,r) \cap M) = 2r$ for all $r>0$. We will derive a contradiction. We distinguish two cases according to the relative location of the second line segment $\gamma_2$.

In the first case, we assume that there exists a radius $r_0>0$ so that the closed $d_H$-ball centered at $y$ with radius $r_0$ meets $\gamma_2$, but does not contain $x$. For some sufficiently small $\delta>0$, if we set $r=r_0+\delta$, then
$$
B_H(y,r) \cap M = E_1 \cup E_2,
$$
where $E_1$ is a line segment along $\gamma_1$ of length $2r$ and $E_2$ is a line segment along $\gamma_2$ of positive length. Then $\mu(B_H(y,r)\cap M) = \mu(E_1) + \mu(E_2) > 2r$ which is a contradiction.

In the second case, we assume that $B_H(y,r_0) \cap \gamma_2 = \emptyset$, where $r_0 = d_H(x,y)$. Again, for some sufficiently small $\delta>0$ we let $r=r_0+\delta$. Then $\partial B_H(y,r) \cap \gamma_2$ contains a single point, which we denote by $z$. Moreover, $B_H(y,r) \cap M = E_1 \cup E_2$, where $E_1$ is a line segment on $\gamma_1$ of length $d_H(x,y)+r$ and $E_2$ is a line segment on $\gamma_2$ of length $d_H(x,z)$. It suffices to prove that
$$
d_H(x,y) + r + d_H(x,z) \ne 2r.
$$
Let $d_1 = d_H(x,y)$ and $d_2 = d_H(x,z)$; we must show that $d_1 + d_2 \ne r = d_H(y,z)$. Without loss of generality, we may assume that $x = o = (0,0)$, $y = (-d_1,0,0)$, and $z=(d_2\cos\theta,d_2\sin\theta,0)$ for some $\theta \in (0,2\pi)$. We compute
\begin{equation*}\begin{split}
r^4 &= d_H(y,z)^4 \\
&= ((d_1 + d_2 \cos\theta)^2 + (d_2\sin\theta)^2)^2 + 4d_1d_2\sin^2\theta \\
&= d_1^4 + 4d_1^3d_2 \cos\theta + 6d_1^2d_2^2 + 4d_1d_2^3 \cos\theta + d_2^4.
\end{split}\end{equation*}
Since $\theta \ne 0$ mod $2\pi$, the latter expression is strictly less than $(d_1+d_2)^4$. Hence $r<d_1+d_2$ and the proof of the lemma is complete.
\end{proof}

\begin{proof}[Proof of Proposition \ref{propA}]
To simplify the computations we assume that $\Sigma$ is parameterized by the C--C arc length. Thus, if $\Sigma(s) = (x_1(s),x_2(s),x_3(s))$ the following ODEs are satisfied:
\begin{equation}\label{ODE1}
\dot{x}_3 + 2x_1\dot{x}_2-2x_2\dot{x}_1 = 0
\end{equation}
and
\begin{equation}\label{ODE2}
(\dot{x}_1)^2+(\dot{x}_2)^2 = 1.
\end{equation}
First, we assume that $s=0$, $\Sigma(0) = o = (0,0,0)$ and $\dot\Sigma(0) = (1,0,0)$. Theorem \ref{th:magnani-area-formula} implies that
\begin{equation}\label{proofAeqn1}
\cS^1(B_H(o,r) \cap \Sigma) = \int_{\Sigma^{-1}(B_H(o,r))} |(\dot\Sigma(s))_{(1)}| \, ds\,.
\end{equation}
Since $\Sigma$ is parameterized by the C--C arc length, $|(\dot\Sigma(s))_{(1)}| = ((\dot{x}_1)^2+(\dot{x}_2)^2)^{1/2} = 1$ and so \eqref{proofAeqn1} equals
$$
\cH^1(\Sigma^{-1}(B_H(o,r))).
$$
Let us expand the components of $\gamma = \pi \circ \Sigma$ in power series about $s=0$. Since the curve is $C^3$ we have
\begin{equation*}\begin{split}
x_1(s) &= s + \tfrac12 a_2s^2 + \tfrac16 a_3s^3 + o(s^3) \\
x_2(s) &= \tfrac12 b_2s^2 + \tfrac16 b_3s^3 + o(s^3) \,.
\end{split}\end{equation*}
The unit speed normalization \eqref{ODE2} gives $1+2a_2s+(a_2^2+a_3+b_2^2)s^2+o(s^2)=1$ for all $s$, whence $a_2=0$ and $a_3=-b_2^2$. We rewrite the equation for $x_1(s)$ as follows:
$$
x_1(s) = s - \tfrac16 b_2^2 s^3 + o(s^3) \,.
$$
Since $\Sigma$ is horizontal, \eqref{ODE1} implies that
$$
x_3(s) = -\tfrac13 b_2 s^3 + o(s^3).
$$
Now $s \in \Sigma^{-1}(B_H(o,r))$ if and only if $|\Sigma(s)|_H \le r$; inserting the power series expansions for the components of $\Sigma$ gives
\begin{equation}\label{s-from-r}
s^4 - \tfrac1{18} b_2^2 s^6 + o(s^6) \le r^4.
\end{equation}
A monotonicity argument (compare Lemma \ref{rho_0} for a similar convexity argument in higher dimensions) ensures the existence of unique positive and negative solutions $s_+=s_+(r)$ and $s_-=s_-(r)$ to the equation $|\Sigma(s)|_H = r$; inverting \eqref{s-from-r} yields
\begin{equation}\label{r-from-s-1}
s_+ = r + \tfrac{1}{72}b_2^2 r^3 + o(r^3).
\end{equation}
and
\begin{equation}\label{r-from-s-1a}
s_- = - r - \tfrac{1}{72}b_2^2 r^3 + o(r^3).
\end{equation}
Then
$$
\cS^1(B_H(o,r) \cap \Sigma) = s_+ - s_- = 2r + \tfrac{1}{36} b_2^2 r^3 + o(r^3)
$$
and we recall that $b_2 = \ddot{x}_2(0)$.

To conclude the proof we remove the assumptions $\Sigma(0) = o$ and $\dot\Sigma(0) = (1,0,0)$.
Indeed we may consider left translations and rotations
$R_\theta:\Heis\to\Heis$ about the $x_3$-axis by angle $\theta$.
The matrix of $R_\theta$ has the form \eqref{eq:Rotations}, so $R_\theta$
is an isometry with respect to $d_H$, along with left translations.
In addition, $R_\theta$ is also a Lie group homomorphism, so it sends metric balls into metric balls. Since this property is also satisfied by left translations, we conclude that
both of these families of mappings preserve the spherical Hausdorff measure.
We also observe that nonhorizontal directions are preserved, since both isometries and left translations are contact diffeomorphisms.

Let $\Sigma$ be an arbitrary horizontal curve and let $x=\Sigma(0)$. Define $\overline\Sigma := R_\theta^{-1} \circ \ell_x^{-1} \circ \Sigma$, where $\ell_x:\Heis\to\Heis$ denotes left translation by $x$ and $\theta$ is chosen so that $\arg(\pi\circ\ell_x^{-1}\circ\Sigma)'(0) = \theta$. Then $\dot{\overline\Sigma}(0) = (1,0,0)$. If we write $\overline\Sigma = (\overline{x_1},\overline{x_2},\overline{x_3})$ then
$$
\ddot{\overline{x_2}}(0) = \dot{x}_1(0)\ddot{x}_2(0) - \dot{x}_2(0) \ddot{x}_1(0) = k_{\gamma}(0)
$$
since $\gamma = \pi \circ \Sigma$ is unit speed parameterized. This completes the proof of Proposition \ref{propA}.
\end{proof}

\section{$2$-uniform measures and smooth fully nonhorizontal curves}\label{sec:2-uniform}

In this section we consider $2$-uniform measures. Theorem~\ref{thmB} is a consequence of a power series formula for the spherical $2$-Hausdorff measure of small balls centered at nonhorizontal points on smooth curves. In this case we make use of the first three nonzero terms in the power series following the leading order term. Specifically, we compute the expansion of $\cS^2(B_H(x,r) \cap \Sigma)$ up to the $r^{14}$ term. This expansion yields multiple differential equations satisfied by the speed of the planar projection of $\Sigma$, which we analyze in order to show that $\Sigma$ must be vertical.

In order to simplify the statement of the formula we introduce some notation. We consider a smooth curve $\Sigma:I \to \Heis$ and we denote by $\gamma = \pi \circ \Sigma$ the planar projection of $\Sigma$. We do not assume that $\gamma$ is parameterized by arc length. The standard Euclidean inner product and symplectic form in $\R^2$ are denoted $g(\vec{v},\vec{w}) = v_1w_1+v_2w_2$ and $\omega(\vec{v},\vec{w}) = v_1w_2-v_2w_1$ respectively. We denote by $\dot\gamma = \gamma^{(1)}$, $\ddot\gamma=\gamma^{(2)}$, $\dddot\gamma=\gamma^{(3)}$, and so on, and we write
$$
g_{ij} = g(\gamma^{(i)},\gamma^{(j)}) \quad \mbox{and} \quad \omega_{ij} = \omega(\gamma^{(i)},\gamma^{(j)}).
$$
We denote by $\sigma:=|\dot\gamma| = g_{11}^{1/2}$ the speed of $\gamma$, and by $k=g_{11}^{-3/2}\omega_{12}$ its planar curvature.

We will consider weighted homogeneous polynomials in the quantities $g_{ij}$ and $\omega_{ij}$, where the weight associated to index $i$ is $2i-1$. Thus, for instance, $g_{11}$ has weight $2$ while $\omega_{12}$ has weight $4$.

\begin{proposition}\label{propB}
Let $\Sigma$ be a $C^{6}$ smooth nonhorizontal curve in $\Heis$, parameterized by homogeneous arclength. Let $x = \Sigma(0)$ and set $\gamma = \pi \circ \Sigma$. Then
\begin{equation}\label{eqnB}
\cS^2(B_H(x,r)\cap\Sigma) = 2r^2 + \bb_1(x) r^6 + \bb_2(x) r^{10} + \bb_3(x) r^{14} + o(r^{14}),
\end{equation}
where $\bb_m(x)$ is a weighted homogeneous polynomial in the quantities $g_{ij}(0)$ and $\omega_{ij}(0)$ of total weight $4m$.
\end{proposition}

The coefficients in \eqref{eqnB} have the following explicit forms:
\begin{equation}\label{eqnB1}
\bb_1(x) = \biggl.-g_{11}^2 + \tfrac23 \omega_{12}\biggr|_{s=0} \, ,
\end{equation}
\begin{equation}\label{eqnB2}
\bb_2(x) = \biggl.\tfrac74 g_{11}^4 - \tfrac73 g_{11}^2 \omega_{12} + \tfrac16 \omega_{12}^2 - \tfrac32 g_{12}^2 - \tfrac23 g_{11}g_{13} + \tfrac1{15} \omega_{23} + \tfrac1{10}\omega_{14}\biggr|_{s=0} \, ,
\end{equation}
and
\begin{equation}\label{eqnB3}\begin{split}
\bb_3(x) = &-\tfrac{33}{16} g_{11}^6 + \tfrac{33}{8} g_{11}^4 \omega_{12} - \tfrac{11}{8} g_{11}^2 \omega_{12}^2 - \tfrac{11}{36} \omega_{12}^3 + \tfrac{63}{8} g_{11}^2 g_{12}^2 - \tfrac1{32} g_{22}^2 \\
& \quad - \tfrac94 g_{12}^2\omega_{12} + \tfrac32 g_{11}^3g_{13} -\tfrac25 g_{11}^2 \omega_{23} - \tfrac54 g_{11} (\omega_{12}g_{13} + g_{12}\omega_{13}) \\
& \qquad - \tfrac7{90} g_{12}g_{23} - \tfrac{31}{180} g_{22}g_{13} + \tfrac1{12} (\omega_{13}^2 - g_{13}^2) + \tfrac{11}{120} \omega_{12}\omega_{14} - \tfrac18 g_{12}g_{14} \\
& \biggl.\quad \qquad - \tfrac9{40} g_{11}^2 \omega_{14} - \tfrac1{60} g_{11}g_{15} + \tfrac1{504} \omega_{34} + \tfrac1{280} \omega_{25} + \tfrac1{504} \omega_{16}\biggr|_{s=0} \, .
\end{split}\end{equation}
By analyzing the various differential equations arising by setting the right hand sides of \eqref{eqnB1}, \eqref{eqnB2} and \eqref{eqnB3} equal to zero, we deduce the following consequence.

\begin{proposition}\label{thmB1}
Let $\Sigma$ be a fully nonhorizontal real analytic curve which is contained in the support of a $2$-uniform measure in $\Heis$. Then $\Sigma$ is contained in a vertical line.
\end{proposition}

We first establish Theorem \ref{thmB} conditional upon the validity of Propositions \ref{propB} and \ref{thmB1}. Proofs of Propositions \ref{propB} and \ref{thmB1} are deferred to the end of the section. We begin with a lemma.

\begin{lemma}\label{lemP}
Any real analytic curve contained in the support of a $2$-uniform measure contains no horizontal points.
\end{lemma}

\begin{proof}
Let $\Sigma:I \to \Heis$ be a real analytic curve contained in the support of a $2$-uniform measure $\mu$. Analyticity implies that there can be no non-isolated horizontal points in $\Sigma$. Assume that $0 \in I$ with $\Sigma(0) = o$, that $\Sigma(s)$ is nonhorizontal for all $s \ne 0$ in $I$, and that $\Sigma(0)$ is horizontal. We claim that $\lim_{r\to 0} r^{-2} \mu(B_H(o,r) \ne 2$; recall that this limit is equal to $2$ at nonhorizontal points by Lemma \ref{zeroth1}. Since $\mu$ is $2$-uniform, this leads to a contradiction.

To this end, we compute
\begin{equation}\label{1}
\frac{\mu(B_H(0,r))}{r^2} = \frac1{r^2} \int_{\Sigma^{-1}(B_H(0,r))} |\vartheta(\dot\Sigma(s))| \, ds.
\end{equation}
To simplify later computations we perform an initial rotation about the vertical axis and assume that $\dot{x}_2(0)=0$, where $\Sigma=(x_1,x_2,x_3)$. We next make the change of variables $s = r \tau$, which transforms the right hand side of \eqref{1} into
$$
\frac1r \int_{r^{-1} \Sigma^{-1}(B_H(0,r))} |\vartheta(\dot\Sigma(r\tau))| \, d\tau.
$$
Note that $\tau \in r^{-1} \Sigma^{-1}(B_H(0,r))$ if and only if $\Sigma(r\tau) \in B_H(0,r)$, which in turn occurs if and only if
$$
\left( \frac{x_1(r\tau)}{r} , \frac{x_2(r\tau)}{r} , \frac{x_3(r\tau)}{r^2} \right) \in B_H(0,1).
$$
We conclude that
$$
\lim_{r\to 0} \frac{\mu(B_H(0,r))}{r^2} = \lim_{r\to 0} \int_{r^{-1} \Sigma^{-1}(B_H(0,r))} \frac1r |\vartheta(\dot\Sigma(r\tau))| \, d\tau,
$$
and since $\vartheta(\dot\Sigma(0)) = 0$ and $\vartheta(\ddot\Sigma(0)) = \ddot{x}_3(0)$, we obtain
$$
\lim_{r\to 0} \frac{\mu(B_H(0,r))}{r^2} = |\vartheta(\ddot\Sigma(0))| \, \int_{J_0} |\tau| \, d\tau = |\ddot{x}_3(0)| \, \int_{J_0} |\tau| \, d\tau,
$$
where
$$
J_0 = \{ \tau \, : \, (\dot{x}_1(0)\tau,0,\tfrac12 \ddot{x}_3(0)\tau^2) \in B_H(0,1) \}.
$$
Since $\tau \in J_0$ if and only if $(\dot{x}_1(0)^4 + \tfrac14 \ddot{x}_3(0)^2) \tau^4 < 1$, we conclude that
\begin{equation*}\begin{split}
\int_{J_0} |\tau| \, d\tau = 2 \int_0^{(\dot{x}_1(0)^4 + \tfrac14 \ddot{x}_3(0)^2)^{-1/4}} \tau \, d\tau = \frac1{(\dot{x}_1(0)^4 + \tfrac14 \ddot{x}_3(0)^2)^{1/2}}
\end{split}\end{equation*}
and hence
$$
\lim_{r\to 0} \frac{\mu(B_H(0,r))}{r^2} = 2 \sqrt{\frac{\tfrac14 \ddot{x}_3(0)^2}{\dot{x}_1(0)^4 + \tfrac14 \ddot{x}_3(0)^2}} < 2.
$$
It follows that if $\mu$ is $2$-uniform, then $\spt\mu$ contains no horizontal points.
\end{proof}

\begin{proof}[Proof of Theorem \ref{thmB}]
Let $\mu$ be a $2$-uniform measure on $\Heis$.
As in the proof of Theorem~\ref{thmA}, due to \cite[Proposition~3.1]{ct:marstrand} $M=\spt\mu$ is a real analytic variety. Lojasiewicz's Structure Theorem, together with Theorem~\ref{zeroth2} where $\ovm=2$ and formula \eqref{eq:area-formula}, implies that $M = M_{(0)} \cup M_{(1)}$ is a union of countably many analytic submanifolds whose dimensions are at most one. Moreover, the open subset of points with degree two in any such curve is dense. By the definition of degree, \eqref{d:degree}, the velocity vector is not horizontal at these points. Lemma \ref{lemP} ensures that any such curve is fully nonhorizontal.

According to Proposition \ref{thmB1}, each of the (fully nonhorizontal) curves in $M_{(1)}$ is purely vertical. By $2$-uniformity, there can be no isolated points in $M_{(0)}$.  The full support must therefore consist of nontrivial vertical line segments (possibly infinite at one or both ends). Endpoints are not allowed (since the measures of small balls at such points would be too small). Thus $\spt\mu$ is a union of affine vertical lines. Again using $2$-uniformity, we deduce that $\spt\mu$ must be connected. Hence the support is a single vertical line and $\mu$ is proportional to $\cS^2_H$ restricted to the support. This completes the proof.
\end{proof}

We now turn to the proof of Proposition \ref{propB} and its corollary, Proposition \ref{thmB1}.

\begin{proof}[Proof of Proposition \ref{propB}]
Let $\Sigma=(x_1,x_2,x_3)$ and $x = \Sigma(0)$ be as in the statement of the proposition. After a left translation, we may assume that $\Sigma(0) = o = (0,0,0)$. We expand $x_1(s)$ and $x_2(s)$ to sixth order:
\begin{equation*}\begin{split}
x_1(s) &= \sum_{m=1}^6 \tfrac1{m!}a_m s^m + o(s^6) \\
x_2(s) &= \sum_{m=1}^6 \tfrac1{m!}b_m s^m + o(s^6)
\end{split}\end{equation*}
and use the fact that $\Sigma$ is parameterized by homogeneous arclength to determine a power series expansion for $x_3(s)$. From the differential equation
$$
\dot{x}_3 + 2x_1\dot{x}_2 - 2x_2\dot{x}_1 = 1
$$
we find that
$$
x_3(s) = s + \sum_{m=3}^7 c_ms^m + o(s^7)
$$
with
\begin{align*}
c_3 &= \tfrac13(a_2b_1-a_1b_2), \\
c_4 &= \tfrac16(a_3b_1-a_1b_3), \\
c_5 &= \tfrac1{20}(a_4b_1-a_1b_4) + \tfrac1{30}(a_3b_2-a_2b_3), \\
c_6 &= \tfrac1{90}(a_5b_1-a_1b_5) + \tfrac1{72}(a_4b_2-a_2b_4),
\end{align*}
and
$$
c_7 = \tfrac1{504}(a_6b_1-a_1b_6) + \tfrac1{280}(a_5b_2-a_2b_5) + \tfrac1{504}(a_4b_3-a_3b_4).
$$
Theorem \ref{th:magnani-area-formula} implies that
\begin{equation}\label{proofBeqn1}
\cS^2(B_H(o,r) \cap \Sigma) = \int_{\Sigma^{-1}(B_H(o,r))} |(\dot\Sigma(s))_{(2)}| \, ds\,.
\end{equation}
Note that
$$
\dot\Sigma(s) = \dot{x}_1(s) X_1 + \dot{x}_2(s) X_2 + \vartheta(\dot\Sigma(s)) X_3
$$
and so
$$
(\dot\Sigma(s))_{(2)} = \vartheta(\dot\Sigma(s)) = 1.
$$
It suffices to compute the interval of definition
$$
\Sigma^{-1}(B_H(o,r)) = (s_-(r),s_+(r)).
$$
The endpoints $s_\pm(r)$ are determined by solving the equation
$$
(x_1(s)^2+x_2(s)^2)^2 + x_3(s)^2 = r^4.
$$
We expand the left hand side in a series in $s$ and invert. Omitting extensive algebraic manipulations and making use of the notation $g_{ij},\omega_{ij}$ introduced at the beginning of this section, we record the following formulas for $s_\pm(r)$:
\begin{equation*}\begin{split}
s_\pm(r) = &\pm r^2 \pm \biggl( -\tfrac12g_{11}^2 + \tfrac13\omega_{12} \biggr) r^6 + \biggl( -g_{11}g_{12} + \tfrac16g_{13} \biggr) r^8 \\
&\qquad \pm \biggl( \tfrac78g_{11}^4 - \tfrac76g_{11}^2\omega_{12} + \tfrac1{12}\omega_{12}^2 - \tfrac34g_{12}^2 -\tfrac13g_{11}g_{13} + \tfrac1{30}\omega_{23} + \tfrac1{20}\omega_{14} \biggr) r^{10} \\
&\qquad \qquad + C_{12}(a_1,\ldots,a_5,b_1,\ldots,b_5)r^{12} \pm C_{14}(a_1,\ldots,a_6,b_1,\ldots,b_6)r^{14}+ o(r^{14}).
\end{split}\end{equation*}
Here all of the quantities $g_{ij}$ and $\omega_{ij}$ are assumed to be evaluated at $s=0$, and $C_{14}$ is one half of the quantity on the right hand side of \eqref{eqnB3}.\footnote{One approach to computing the coefficients in the formula for $C_{14}$ is to write down an arbitrary linear combination of all weighted homogeneous polynomials of degree $12$ in the $g_{ij}$'s and $\omega_{ij}$'s, compare this to the coefficient of $r^{14}$ in the series, and solve for the coefficients.} We omit the value of the coefficient $C_{12}$ as it does not appear in the final formula. We now simply compute
$$
\cS^2(B_H(o,r) \cap \Sigma) = s_+(r) - s_-(r)
$$
i.e.\ $\cS^2(B_H(o,r) \cap \Sigma)$ is equal to
\begin{equation*}\begin{split}
& 2r^2 + \biggl( -g_{11}^2 + \tfrac23\omega_{12} \biggr) r^6 + \biggl( \tfrac74g_{11}^4 - \tfrac73g_{11}^2\omega_{12} + \tfrac1{6}\omega_{12}^2 - \tfrac32g_{12}^2 -\tfrac23g_{11}g_{13} + \tfrac1{15}\omega_{23} + \tfrac1{10}\omega_{14} \biggr) r^{10} \\
&\quad + 2C_{14}(a_1,\ldots,a_6,b_1,\ldots,b_6)r^{14} + o(r^{14}).
\end{split}\end{equation*}
In order to remove the assumption that $\Sigma(0)=o$ we perform a left translation by the point $x$, and observe that the action of a Heisenberg left translation in the first two coordinates is simply a Euclidean translation, which preserves the quantities $g_{ij}$ and $\omega_{ij}$. This completes the proof.
\end{proof}

\begin{proof}[Proof of Proposition~\ref{thmB1}]
Let $\mu$ be a $2$-uniform measure, and let $\Sigma$ be a fully nonhorizontal curve contained in the $1$-dimensional stratum of $\spt\mu$. We may assume without loss of generality that $\Sigma$ is parameterized by homogeneous arclength, and that $\mu$ agrees with the spherical $2$-Hausdorff measure $\cS^2_H$ restricted to $\Sigma$. As usual, we denote by $\gamma$ the planar projection of $\Sigma$. By analyticity, either $\Sigma$ is purely vertical, or $\dot\gamma$ is nonzero on a dense open subset of the parameterizing interval. We will show that the latter case cannot occur.

Let us suppose that $\gamma$ is nondegenerate on a nonempty open interval. Since $\cS^2 \restrict \Sigma$ is a (locally) $2$-uniform measure, the coefficients $\bb_1,\bb_2,\bb_3$ in \eqref{eqnB} must vanish throughout $\Sigma$. Equating $\bb_1$ to zero shows that the curvature of $\gamma$ is proportional to its speed:
\begin{equation}\label{EQ1}
\sigma:=|\dot\gamma| = \frac23 k \,.
\end{equation}
We will state the remaining two differential equations in terms of the speed $\sigma$. We show that the only solution to this pair of equations is $\sigma=0$. This contradicts our previous assumption.

\begin{lemma}\label{lemQ}
The differential equation obtained by setting $\bb_2=0$ is
\begin{equation}\label{PDE-B2}
-\frac16 \sigma \ddot\sigma + \dot\sigma^2 + \frac18 \sigma^6 = 0.
\end{equation}
while the differential equation obtained by setting $\bb_3=0$ is
\begin{equation}\begin{split}\label{PDE-B3}
-\frac4{39} \sigma^3 \ddddot\sigma + \frac{28}{39} \sigma^2 \dot\sigma \dddot\sigma + \frac{10}{13} \sigma^2 \ddot\sigma^2 + \frac{16}3 \sigma \dot\sigma^2 \ddot\sigma + \frac{10}{13} \sigma^7 \ddot\sigma + \dot\sigma^4 + \frac72 \sigma^6 \dot\sigma^2 - \frac9{208} \sigma^{12} = 0.
\end{split}\end{equation}
\end{lemma}

\

We postpone the proof of this lemma. Assume that the speed $\sigma$ of $\gamma$ satisfies both \eqref{PDE-B2} and \eqref{PDE-B3}. We use \eqref{PDE-B2} to reduce \eqref{PDE-B3} to the following first-order ODE for $\sigma$:
\begin{equation}\label{PDE-B4}
81\sigma^6 + 32\dot{\sigma}^2 = 0.
\end{equation}
Hence $\sigma$ must be identically equal to zero. This completes the proof of Proposition \ref{thmB1}, modulo the lemma.
\end{proof}

It remains to prove Lemma \ref{lemQ}.

\begin{proof}[Proof of Lemma \ref{lemQ}]
We write $\sigma = |\dot\gamma| = \sqrt{\dot{x}_1^2+\dot{x}_2^2}$ for the speed of $\gamma$; equation \eqref{EQ1} reads $\sigma = \tfrac23k$.
In order to simplify the equations $\bb_2=0$ and $\bb_3=0$, we derive formulas for the quantities $g_{ij}$ and $\omega_{ij}$ in terms of $\sigma$. We use the following differential and algebraic identities:
\begin{equation}\label{g-ij-prime}
g_{ij}' = g_{i,j+1}+g_{i+1,j},
\end{equation}
\begin{equation}\label{omega-ij-prime}
\omega_{ij}' = \omega_{i,j+1} + \omega_{i+1,j},
\end{equation}
\begin{equation}\label{g2-omega2}
g_{ij}^2 + \omega_{ij}^2 = g_{ii}g_{jj},
\end{equation}
and
\begin{equation}\label{cocycle-type-identity}
\omega_{ij}\omega_{jk} = g_{ij}g_{jk} - g_{ik}g_{jj}.
\end{equation}
Repeated use of these identities yields explicit formulas for the $g_{ij}$'s and $\omega_{ij}$'s.
\begin{itemize}
\item[(1)] $g_{11} = \sigma^2$.

\noindent This is a restatement of the definition of $\sigma$.
\item[(2)] $\omega_{12} = \tfrac32 \sigma^4$.

\noindent This follows from the definition of curvature and equation \eqref{EQ1}.
\item[(3)] $g_{12} = \sigma \dot\sigma$.

\noindent This follows from (1) and an application of \eqref{g-ij-prime} with $i=j=1$.
\item[(4)] $g_{22} = \dot\sigma^2 + \tfrac94 \sigma^6$.

\noindent This follows from (1), (2) and (3) and an application of \eqref{g2-omega2} with $i=1$ and $j=2$.
\item[(5)] $g_{23} = \dot\sigma\ddot\sigma + \tfrac{27}4\sigma^5\dot\sigma$.

\noindent This follows from (4) and an application of \eqref{g-ij-prime} with $i=j=2$.
\item[(6)] $\omega_{13} = 6 \sigma^3 \dot\sigma$.

\noindent This follows from (2) and an application of \eqref{omega-ij-prime} with $i=1$ and $j=2$. (Note that $\omega_{22}=0$.) Subsequent identities are proved in much the same way, and we omit any further explanation.
\item[(7)] $g_{13} = \sigma \ddot\sigma - \tfrac94 \sigma^6$.

\item[(8)] $g_{33} = (\ddot\sigma - \tfrac94\sigma^5)^2 + 36 \sigma^4 \dot\sigma^2$.

\item[(9)] $g_{34} = (\ddot\sigma -\tfrac94 \sigma^5)(\dddot\sigma - \tfrac{45}4 \sigma^4 \dot\sigma) + 36 \sigma^4 \dot\sigma \ddot\sigma + 72 \sigma^3 \dot\sigma^3$.

\item[(10)] $g_{14} = \sigma \dddot\sigma - \tfrac{81}4 \sigma^5 \dot\sigma$.

\item[(11)] $g_{24} = \dot\sigma \dddot\sigma + \tfrac{45}4 \sigma^5 \ddot\sigma - \tfrac94 \sigma^4 \dot\sigma^2 - \tfrac{81}{16} \sigma^{10}$.

\item[(12)] $\omega_{23} = -\tfrac32 \sigma^3 \ddot\sigma + 6 \sigma^2 \dot\sigma^2 + \tfrac{27}8 \sigma^8$.

\item[(13)] $\omega_{14} = \tfrac{15}2 \sigma^3 \ddot\sigma + 12 \sigma^2 \dot\sigma^2 - \tfrac{27}8 \sigma^8$.
\end{itemize}
Identity (13) follows either from (1), (2), (3), (10), (11) and an application of \eqref{cocycle-type-identity} with $i=2$, $j=1$ and $k=4$, or from (1), (6), (7), (9), (10) and an application of \eqref{cocycle-type-identity} with $i=3$, $j=1$ and $k=4$. Using the preceding list of identities, we simplify the formula for $\bb_2$ in \eqref{eqnB2}. The equation $\bb_2 = 0$ reduces to \eqref{PDE-B2}.

We continue with the preceding list of identities.
\begin{itemize}
\item[(14)] $g_{44} = (\dddot\sigma-\tfrac{81}4\sigma^4\dot\sigma)^2+(\tfrac{15}2\sigma^2\ddot\sigma+12\sigma\dot\sigma^2-\tfrac{27}8\sigma^7)^2$.

\item[(15)] $g_{15} = \sigma \ddddot\sigma - \tfrac{63}2\sigma^5\ddot\sigma-99\sigma^4\dot\sigma^2+\tfrac{81}{16}\sigma^{10}$.

\item[(16)] $\omega_{24} = -\tfrac32 \sigma^3 \dddot\sigma + \tfrac{15}2\sigma^2\dot\sigma\ddot\sigma + 12\sigma \dot\sigma^3 + 27 \sigma^7\dot\sigma$.

\item[(17)] $\omega_{34} = -6\sigma^2\dot\sigma\dddot\sigma + \tfrac{15}2 \sigma^2\ddot\sigma^2 + 12 \sigma \dot\sigma^2\ddot\sigma - \tfrac{81}2 \sigma^7 \ddot\sigma + \tfrac{189}2 \sigma^6 \dot\sigma^2 + \tfrac{243}{32}\sigma^{12}$.

\item[(18)] $\omega_{25} = \tfrac32 \sigma^3 \ddddot\sigma + 9 \sigma^2\dot\sigma\dddot\sigma + 39 \sigma \dot\sigma^2 \ddot\sigma + \tfrac{135}2 \sigma^7 \ddot\sigma + 12 \dot\sigma^4 + \tfrac{189}2 \sigma^6 \dot\sigma^2 - \tfrac{243}{32} \sigma^{12}$.

\item[(19)] $\omega_{15} = 9 \sigma^3 \dddot\sigma + 39 \sigma^2\dot\sigma\ddot\sigma + 12 \sigma \dot\sigma^3 - 54 \sigma^7 \dot\sigma$.

\item[(20)] $\omega_{16} = \tfrac{21}2 \sigma^3 \ddddot\sigma + 57 \sigma^2\dot\sigma\dddot\sigma + 75 \sigma \dot\sigma^2 \ddot\sigma + 39 \sigma^2 \ddot\sigma^2 - \tfrac{243}2 \sigma^7 \ddot\sigma - \tfrac{945}2 \sigma^6 \dot\sigma^2 + \tfrac{243}{32} \sigma^{12}$.
\end{itemize}
Using these identities, we simplify the formula for $\bb_3$ in \eqref{eqnB3}. The equation $\bb_3 = 0$ reduces to \eqref{PDE-B3}.
\end{proof}

\section{$3$-uniform measures and smooth fully noncharacteristic surfaces}\label{sec:3-uniform}

The geometric equation \eqref{PDE-C} is an immediate consequence of the following power series formula for the local measure of small Kor\'anyi balls at noncharacteristic points of smooth surfaces.

\begin{theorem}\label{theoC}
If $\Sigma$ is a $C^4$ smooth surface in $\Heis$ and $x \in \Sigma$
is a noncharacteristic point, then the following expansion holds
\begin{equation}\label{eqnC}
\cS^3(B_H(x,r)\cap\Sigma) = \bc_0 r^3 + \bc_1(x) r^5 + o(r^{5}),
\end{equation}
where $\bc_0$ is a geometric constant and $\bc_1(x)$ agrees, up to a fixed multiplicative constant, with the quantity
$$
\cH_0(x)^2 + \frac32 \cP_0(x)^2 + 4 (\vec{e}_1)_x(\cP_0)(x) \, .
$$
\end{theorem}

\begin{corollary}\label{coroC}
Let $\Sigma$ be a surface contained in the support of a $3$-uniform measure. Then the equation
\begin{equation}\label{eq:coroC}
\cH_0^2 + \tfrac32 \cP_0^2 + 4 (\vec{e}_1)(\cP_0) = 0
\end{equation}
holds in all noncharacteristic patches of $\Sigma$.
\end{corollary}

We recall that $\cH_0$ denotes the horizontal mean curvature of $\Sigma$, while $\cP_0$ denotes the imaginary curvature of Arcozzi--Ferrari (curvature of the metric normal). See section \ref{sec:background} for the definitions. The leading coefficient $\bc_0$ in \eqref{eqnC} agrees with the quantity $\omega_H$ in \eqref{eq:omegaH}, which is the area of the unit Kor\'anyi ball in a vertical plane.

As in previous sections, we will first prove Theorem \ref{thmC} under the assumption that Theorem \ref{theoC} holds, after which we prove Theorem \ref{theoC}.
Before beginning the proof of Theorem \ref{thmC}, we pause to answer a question which may have occurred to the reader, namely, why we impose the hypothesis that all points in the support be non-characteristic. Recall that in the corresponding proof in the $2$-uniform case, Lemma \ref{lemP} ruled out the existence of horizontal points on the supports of $2$-uniform measures, by showing that the local density coefficient at horizontal points differs from that at non-horizontal points. We will see that the corresponding statement for characteristic vs.\ non-characteristic points on surfaces is not true. A counterexample can be found already within the class of circular paraboloids. We also study the validity of our PDE in the noncharacteristic region of such paraboloids.

\begin{example}\label{ex:paraboloids}
Let us consider the surface $\Sigma$ parameterized over the plane as follows:
$$
\Phi(\eta) = (\eta,\tan\alpha\,|\eta|^2), \qquad \eta=(\eta_1,\eta_2) \in \R^2.
$$
Following the model in the $2$-uniform case, we compute
\begin{equation}\label{1a}
\frac{\mu(B_H(0,r))}{r^3} = \frac1{r^3} \int_{\Phi^{-1}(B_H(0,r) \cap \Sigma)} |(\Phi_1 \wedge \Phi_2)_{(3)}| \, d\eta.
\end{equation}
Here, as before, $\Phi_j$ denotes the partial derivative of $\Phi$ with respect to $\eta_j$. Since the paraboloid $\Sigma$ is invariant under dilation, the preceding ratio is constant in $r$ and we need only compare its value to the corresponding value for a vertical plane, namely, the quantity $\omega_H$ described in Lemma \ref{zeroth1} and whose value is given in \eqref{eq:omegaH}.

The change of variables $\eta = r \xi$ transforms the right hand side of \eqref{1a} into
$$
\frac1r \int_{r^{-1} \Phi^{-1}(B_H(0,r) \cap \Sigma)} |(\Phi_1(r\xi)\wedge\Phi_2(r\xi))_{(3)}| \, d\xi.
$$
Note that $\xi \in r^{-1} \Phi^{-1}(B_H(0,r))$ if and only if $\Phi(r\xi) \in B_H(0,r)$, which in turn occurs if and only if
\begin{equation*}
|\xi| < \sqrt{\cos\alpha}\,.
\end{equation*}
Furthermore, $\Phi_1(r\xi) = X + 2r(\tan\alpha\,\xi_1-\xi_2)T$ and $\Phi_2(r\xi) = Y + 2r(\tan\alpha\,\xi_2+\xi_1)T$, and so
$$
(\Phi_1(r\xi)\wedge \Phi_2(r\xi))_{(3)} = 2r(\tan\alpha\,\xi_1-\xi_2)(X\wedge T) - 2r(\tan\alpha\,\xi_1-\xi_2)(Y\wedge T).
$$
We conclude that
\begin{equation*}\begin{split}
\frac{\mu(B_H(0,r))}{r^3} &= \int_{\{\xi:|\xi|<\sqrt{\cos\alpha}\}} 2 \sqrt{(\tan\alpha\,\xi_2+\xi_1)^2+(\tan\alpha\xi_1-\xi_2)^2} \, d\xi \\
&= 2 \sec\alpha \, \int_{\{\xi:|\xi|<\sqrt{\cos\alpha}\}} |\xi| \, d\xi \\
&=(2\sec\alpha)(2\pi)(\int_0^{\sqrt{\cos\alpha}} \rho^2\,d\rho) =\frac{4\pi}{3} \sqrt{\cos\alpha} \, .
\end{split}\end{equation*}
However, since $\omega_H < \tfrac{4\pi}{3}$ there exists a choice for $\alpha$ such that
\begin{equation}\label{eq:for-alpha-1}
\frac{4\pi}{3} \sqrt{\cos\alpha} = \omega_H.
\end{equation}
For this particular paraboloid, we are not able to distinguish characteristic and non-character\-istic points via the local density behavior of the measure.

We now show that the PDE \eqref{eq:coroC} is satisfied for a specific choice of $\alpha$, which is not the same as the value for which \eqref{eq:for-alpha-1} holds. The authors would like to thank Andrea Merlo for pointing out this fact. Let $u(x_1,x_2,x_3) = (\tan\alpha)(x_1^2+x_2^2)-x_3$ be a defining function for $\Sigma_\alpha$. Then $X_1u = 2(\tan\alpha)x_1-2x_2$ and $X_2u = 2(\tan\alpha)x_2+2x_1$, and so
$$
|\nabla_0u| = 2(\sec\alpha)\sqrt{x_1^2+x_2^2}.
$$
The unit horizontal normal is
$$
\vec{n}_0 = (\sin\alpha \, \cI + \cos\alpha \, \cJ)\left( \frac{x_1X_1+x_2X_2}{\sqrt{x_1^2+x_2^2}} \right)
$$
where $\cI$ denotes the fiberwise identity operator acting on the horizontal distribution. The characteristic vector field is
$$
\vec{e}_1 = \cJ(\vec{n}_0) = (-\cos\alpha \, \cI + \sin\alpha \, \cJ)\left( \frac{x_1X_1+x_2X_2}{\sqrt{x_1^2+x_2^2}} \right) \,.
$$
The horizontal mean curvature is
$$
\cH_0 = X_1(\frac{X_1u}{|\nabla_0u|}) + X_2(\frac{X_2u}{|\nabla_0u|}) = \frac{\sin\alpha}{\sqrt{x_1^2+x_2^2}},
$$
while the imaginary curvature is
$$
\cP_0 = \frac{4X_3u}{|\nabla_0u|} = \frac{-2\cos\alpha}{\sqrt{x_1^2+x_2^2}} \,.
$$
We compute
$$
\vec{e}_1(\cP_0) = (-\cos\alpha \, \cI + \sin\alpha \, \cJ)\left( \frac{x_1X_1+x_2X_2}{\sqrt{x_1^2+x_2^2}} \right) \left( \frac{-2\cos\alpha}{\sqrt{x_1^2+x_2^2}} \right) = \frac{-2\cos^2\alpha}{x_1^2+x_2^2} \,.
$$
Finally,
$$
\cH_0^2 + \frac32 \cP_0^2 + 4 \vec{e}_1(\cP_0) = \frac{\sin^2\alpha}{x_1^2+x_2^2} + \frac{6\cos^2\alpha}{x_1^2+x_2^2} - \frac{8\cos^2\alpha}{x_1^2+x_2^2} = \frac{\sin^2\alpha - 2 \cos^2\alpha}{x_1^2+x_2^2}
$$
is identically equal to zero if $\tan^2\alpha = 2$ and never equal to zero for any other $\alpha$.

To conclude this example, we show that the value of $\alpha$ for which the paraboloid $\Sigma_\alpha$ satisfies \eqref{eq:coroC} in the noncharacteristic patch is not the same value of $\alpha$ for which the local density coefficient at the characteristic point agrees with $\omega_H$. From these observations, we deduce that no paraboloid $\Sigma_\alpha$ can be the support of a $3$-uniform measure. This is a partial step towards the proof of Proposition \ref{prop:no-quadratic-t-graphs-as-3-uniform-supports}; the full proof is deferred to the end of this section.

Our PDE is satisfied iff $\tan^2\alpha = 2$ or
$$
\cos\alpha = 3^{-1/2}.
$$
However, according to \ref{eq:for-alpha-1}, the local density coefficient at the characteristic point is equal to $\omega_H$ if and only if
$$
\cos\alpha = \left( \frac{3\omega_H}{4\pi} \right)^2 =\frac1{8\pi^3} \Gamma(\tfrac14)^4.
$$
But $3^{-1/2} \approx 0.57735 \dots$ while $\tfrac1{8\pi^3}\Gamma(\tfrac14)^4 \approx 0.696602 \dots$.
\end{example}

We now turn to the proof of Theorem \ref{thmC}.

\begin{proof}[Proof of Theorem \ref{thmC}]
Let $\mu$ be a $3$-uniform measure on $\Heis$. By \cite{ct:marstrand}, $M = \spt\mu$ is a real analytic variety of topological dimension two and Hausdorff dimension three. The top dimensional stratum $M_{(2)}$ is a countable union of real analytic surfaces. We assume that the lower-dimensional strata of $M$ are empty, and that $M = M_{(2)}$ is a single real-analytic surface which is vertically ruled. In this case, every point on $M$ is noncharacteristic. Since vertically ruled surfaces are characterized by the equation $\cP_0=0$, equation \eqref{PDE-C} in this setting reduces to
$$
\cH_0 = 0.
$$
Thus $M$ is an H-minimal surface. Moreover, the stability criterion of \cite[Theorem 3.2]{dgnp:bernstein2} is verified. (Note that the quantity $\bar\omega$ in that theorem is proportional to $\cP_0$, and hence both $\bar\omega$ and $\cA$ vanish.) In view of \cite[Theorem A]{dgnp:bernstein2} (see also \cite{dgnp:bernstein1} and \cite{hrr:bernstein}), $M$ must be contained in a vertical plane. The $3$-uniformity of the measure $\mu$ ensures that there can be no boundary components, hence $M$ is a single affine vertical plane, and $\mu$ is proportional to the restriction of $\cS^3$ to this plane.
\end{proof}

The remainder of this section is devoted to the proof of Theorem~\ref{theoC}. We proceed in two steps. In the first step, we assume that $x=o$ and $HT_0\Sigma$ is the $e_2$-axis and derive \eqref{eqnC}. In the second step we remove the assumptions on $x$ and $HT_x\Sigma$, obtaining the stated conclusion in the theorem.

Before embarking on the proof, we introduce some additional notation and state a polar coordinate formula for a Kor\'anyi-type norm in $\R^2$.
We define
$$
\nu_K(\eta) := (\eta_1^4+\eta_2^2)^{1/4}, \qquad \qquad \eta=(\eta_1,\eta_2).
$$
Denote by $\B_K^1 = \{ \eta \in  \R^2 : \nu_K(\eta) \le 1\}$ the unit ball for this
norm, and by $\Sph_K^1 = \partial \B_K^1$ the boundary of $\B_K^1$.
Introduce anisotropic dilations $\{\delta_r\,:\,r>0\}$ on $\R^2$ by
\begin{equation}\label{N-dilations}
\delta_r(\eta_1,\eta_2)=(r\eta_1,r^2\eta_2).
\end{equation}
Relative to such dilations, we state the following polar coordinate decomposition.

\begin{lemma}\label{polar-coordinates}
There exists a Radon measure $\sigma$ on $\Sph_K^1$ such that
$$
\int_\Heis h(\eta)\,d\eta = \int_{\Sph^1_K} \, \int_0^\infty
h(\delta_t\xi) t^2 \, dt \, d\sigma(\xi)
$$
for each integrable function $h$ on $\R^2$.
\end{lemma}

This result appears as \cite[Proposition 1.15]{fs:hardy}. For the purposes of this paper we need an explicit representation of the measure $\sigma$.
In section \ref{sec:koranyi-integrals} we provide an explicit
formula for $\sigma$.

\begin{proof}[Proof of Theorem~\ref{theoC}]
We first assume that $x=o\in\Sigma$ is the noncharacteristic point, and that $v_\Sigma$ is a defining function for $\Sigma$ around $o$
such that $X_2v_\Sigma(o)=0$. Clearly, since $o$ is noncharacteristic $X_1v_\Sigma(o)$ is not vanishing and we may select $v_\Sigma$ so that $X_1v_\Sigma(o)>0$.

By the implicit function theorem, setting
$N:=\{\eta_1e_2+\eta_2e_3:\eta_1,\eta_2\in\R\}$
and denoting $\eta=\eta_1e_2+\eta_2e_3$, we may parameterize a neighborhood of $o$ in $\Sigma$ as follows:
\begin{equation}\label{Phi}
\Phi:N \supset \Omega \to \Sigma \subset \Heis,\quad \Phi(\eta) = \eta + \varphi(\eta)e_1,
\end{equation}
where $\varphi:\Omega\to\R$ satisfies $\varphi(0)=0$ and
$\varphi_1(0)=0$ (since $e_2 \in T_o\Sigma$).
To ease the notation, here and in what follows
we write
\[
\varphi_1 = \frac{\partial\varphi}{\partial\eta_1},\quad
\varphi_{11} = \frac{\partial^2\varphi}{\partial\eta_1^2},\quad\text{ etc.}
\]
From \eqref{Phi} we obtain the following expressions for the partial
derivatives of $\Phi$:
$$
\Phi_1 = e_2 + \varphi_1\,e_1 \qquad \Phi_2 = e_3 + \varphi_2\,e_1.
$$
We write these expressions in the intrinsic vector fields
$X_1,X_2,X_3$ by using the identities $e_1=X_1-2x_2X_3$,
$e_2=X_2+2x_1X_3$, $e_3=X_3$. Along $\Phi(\Omega)\subset\Sigma$ we
have $x_1=\varphi(\eta)=\varphi$ and $x_2=\eta_1$. Hence
$$
\Phi_1 = \varphi_1 X_1 + X_2 + 2(\varphi-\eta_1\varphi_1)X_3
$$
and
$$
\Phi_2 = \varphi_2 X_1 + (1-2\eta_1\varphi_2) X_3,
$$
so
\begin{equation}\label{3-component}
(\Phi_1\wedge\Phi_2)_{(3)} = (\varphi_1-2\varphi\varphi_2) X_1\wedge
X_3 + (1-2\eta_1\varphi_2) X_2\wedge X_3.
\end{equation}
By the area formula \eqref{eq:area-formula},
$$
\cS^3(B_H(o,r)\cap\Sigma) = \int_{\Phi^{-1}(B_H(o,r))} |(\Phi_1\wedge\Phi_2)_{(3)}| \, d\eta
$$
and \eqref{3-component} implies that
\begin{equation}\label{to-do}
\cS^3(B_H(o,r)\cap\Sigma) = \int_{\Phi^{-1}(B_H(o,r))}
\sqrt{(\varphi_1-2\varphi\varphi_2)^2+(1-2\eta_1\varphi_2)^2} \, d\eta_1\,d\eta_2
\end{equation}
where
\begin{equation*}\begin{split}
\Phi^{-1}(B_H(o,r))
&= \{ \eta \in N \, : \, \eta + \varphi(\eta)e_1 \in B_H(o,r) \} \\
&= \{ \eta=(\eta_1,\eta_2) \, : \, (\varphi(\eta)^2+\eta_1^2)^2 + \eta_2^2 \le r^4 \}.
\end{split}\end{equation*}
Writing $\eta=\delta_\rho\xi$ as in \eqref{N-dilations} we observe that
$\Phi^{-1}(B_H(o,r))$ can be identified with the set of points
$(\rho,\xi) \in [0,r]\times\Sph_K^1$ such that
$$
(\varphi(\delta_\rho\xi)^2+\rho^2\xi_1^2)^2 + (\rho^2\xi_2)^2 \le r^4.
$$
Since $\xi\in\Sph_K^1$ this inequality is equivalent to the following
$$
\varphi(\delta_\rho\xi)^4 + 2\rho^2\xi_1^2\varphi(\delta_\rho\xi)^2 + \rho^4
\le r^4.
$$

\begin{lemma}\label{rho_0}
There exists $r>0$ sufficiently small such that for each $\xi \in \Sph^1_K$
the equation
\begin{equation}\label{rho_0-equation}
\varphi(\delta_\rho\xi)^4 + 2\rho^2\xi_1^2\varphi(\delta_\rho\xi)^2 + \rho^4
= r^4
\end{equation}
has a unique solution $\rho_0 = \rho_0(\xi,r) \le r$.
\end{lemma}

\begin{proof}
The claim $\rho_0 \le r$ is clear from \eqref{rho_0-equation} since
the terms $\varphi^4$ and $2\rho^2\xi_1^2\varphi^2$ are nonnegative.
To show the existence and uniqueness of $r>0$ we let
$$
F(\rho,\xi) := \varphi(\delta_\rho\xi)^4 +
2\rho^2\xi_1^2\varphi(\delta_\rho\xi)^2 + \rho^4.
$$
An extensive computation (done via {\sc Mathematica}) reveals that
\[
F(0,\xi)=\frac{\partial F}{\partial \rho}(0,\xi)=\frac{\partial^2 F}{\partial \rho^2}(0,\xi)=\frac{\partial^3 F}{\partial \rho^3}(0,\xi)=0\quad\text{ and }\quad \frac{\partial^4 F}{\partial \rho^4}(0,\xi)>0
\]
for all $\xi\in\Sph_K^1$.
This implies that $\frac{\partial^4 F}{\partial \rho^4}(\rho,\xi)\ge c_0>0$
for any $(\rho,\xi)\in[0,\rho_1]\times\Sph_K^1$.
We have proved that $F(\cdot,\xi)$ is strictly increasing on
$[0,\rho_1]$ for all $\xi\in\Sph^1_K$ with size of the image
bounded away from zero. This immediately gives our claim.
\end{proof}

We now expand $\varphi$ by homogeneous polynomials. Let us write
\begin{equation}\label{phi-power-series}
\varphi(\eta) = A_1 + A_2 + A_3 + A_4 + O(|\eta|^5)
\end{equation}
where $A_1 = a_2\eta_2$,
$$
A_2 = a_{11}\eta_1^2 + a_{12}\eta_1\eta_2 + a_{22}\eta_2^2
$$
and
$$
A_3 = a_{111}\eta_1^3 + a_{112}\eta_1^2\eta_2 + a_{122}\eta_1\eta_2^2
+ a_{222}\eta_2^3.
$$
First we compute $\varphi(\delta_\rho\xi)$ and group the resulting
terms according to powers of $\rho$:
$$
\varphi(\delta_\rho\xi) = B_2\rho^2 + B_3\rho^3 + B_4\rho^4 + O(\rho^5),
$$
where
$$
B_2 = a_2\xi_2 + a_{11}\xi_1^2, \qquad \mbox{and} \qquad
B_3 = a_{12}\xi_1\xi_2 + a_{111}\xi_1^3.
$$
Hence
$$
r^4 = \rho_0^4 + 2\rho_0^2\xi_1^2\varphi(\delta_{\rho_0}\xi)^2 +
\varphi(\delta_{\rho_0}\xi)^4
= \rho_0^4 + 2\xi_1^2{B_2}^2 \rho_0^6 + 4\xi_1^2B_2B_3 \rho_0^7 +
O(\rho_0^8).
$$
Taking the inverse gives the following series expansion for $\rho_0$
in terms of $r$:
\begin{equation}\label{rho_0-power-series}
\rho_0(\xi,r) = r - \frac12\xi_1^2{B_2}^2r^3 - \xi_1^2B_2B_3r^4 + O(r^5).
\end{equation}
We now return to \eqref{phi-power-series} and compute a power series
for the integrand in \eqref{to-do}, that is
\[
\sqrt{(\varphi_1-2\varphi\varphi_2)^2+(1-2\eta_1\varphi_2)^2}.
\]
The Taylor expansion of this integrand is as follows
\[
\sqrt{1-4a_2\eta_1+(4a_{11}^2+4a_2^2-4a_{12})\eta_1^2+(4a_{11}(a_{12}-2a_2^2)-8a_{22})\eta_1\eta_2+(a_{12}-2a_2^2)\eta_2^2
  + O(|\eta|^3)}.
\]
Letting $\eta_1=\rho\xi_1$ and $\eta_2=\rho^2\xi_2$ and expanding as a
power series in $\rho$ gives
\begin{equation}\label{integrand-power-series}
\sqrt{(\varphi_1-2\varphi\varphi_2)^2+(1-2\eta_1\varphi_2)^2}(\delta_\rho\xi)
= 1 - 2a_2\xi_1\rho + 2(a_{11}^2-a_{12})\xi_1^2\rho^2 + O(\rho^3).
\end{equation}
We are now ready to compute the integral in \eqref{to-do}. First we
convert to Kor\'anyi polar coordinates:
$$
\cS^3(B_H(o,r)\cap\Sigma) = \int_{\Sph^1_K} \int_0^{\rho_0(\xi,r)}
\sqrt{(\varphi_1-2\varphi\varphi_2)^2+(1-2\eta_1\varphi_2)^2} \,
\rho^2 \, d\rho\, d\sigma(\xi).
$$
Next we insert the power series from \eqref{integrand-power-series}
and integrate term-by-term in $\rho$:
\begin{equation*}\begin{split}
\cS^3(B_H(o,r)\cap\Sigma) &= \int_{\Sph^1_K} \int_0^{\rho_0(\xi,r)}
\biggl(\rho^2 - 2a_2\xi_1\rho^3 + 2(a_{11}^2-a_{12})\xi_1^2\rho^4 +
O(\rho^5)\biggr) \, d\rho\, d\sigma(\xi) \\
&= \int_{\Sph^1_K} \biggl( \frac13 \rho_0^3 - \frac12a_2\xi_1\rho_0^4
+ \frac25(a_{11}^2-a_{12})\xi_1^2\rho_0^5 + O(\rho^6)\biggr) \,
d\sigma(\xi).
\end{split}\end{equation*}
Finally we insert the power series from \eqref{rho_0-power-series}.
The result is
\[
\cS^3(B_H(o,r)\cap\Sigma) =
\int_{\Sph^1_K} \biggl( \frac13 r^3 - \frac12a_2\xi_1 r^4
+ \left( \frac25(a_{11}^2-a_{12}) - \frac12 B_2^2 \right) \xi_1^2 r^5
+ O(r^6) \biggr) \, d\sigma(\xi).
\]
This integral can be written as follows
\[
\left( \frac13 \int_{\Sph^1_K} d\sigma \right)r^3
- \frac12 a_2 \left( \int_{\Sph^1_K} \xi_1 \, d\sigma \right) r^4
+ \left( \int_{\Sph^1_K} \left( \frac25(a_{11}^2-a_{12}) - \frac12
    B_2^2 \right) \xi_1^2 \, d\sigma \right) r^5 + O(r^6).
\]
The integrals over $\Sph^1_K$ can be computed using Corollary
\ref{polar-coordinates-corollary} and Lemma \ref{monomial-integrals}.
First,
\[
\frac13 \int_{\Sph^1_K} d\sigma = \Vol(\B^1_K) = \omega_H.
\]
The value of $\omega_H$ is computed in \eqref{eq:omegaH}. Next, we observe that
\[
\int_{\Sph^1_K} \xi_1 \, d\sigma = 0
\]
by symmetry. For the third term we compute
\begin{equation*}\begin{split}
\int_{\Sph^1_K} &\left( \frac25(a_{11}^2-a_{12}) - \frac12 B_2^2
\right) \xi_1^2 \, d\sigma
\\
&= \frac25(a_{11}^2-a_{12}) \int_{\Sph^1_K} \xi_1^2 \, d\sigma(\xi)
- \frac12 a_2^2 \int_{\Sph^1_K} \xi_1^2 \xi_2^2 \, d\sigma(\xi)
- a_{11}a_2 \int_{\Sph^1_K} \xi_1^4 \xi_2 \, d\sigma(\xi) \\
& \qquad\qquad - \frac12 a_{11}^2 \int_{\Sph^1_K} \xi_1^6 \, d\sigma(\xi).
\end{split}\end{equation*}
Again, the $\xi_1^4\xi_2$ integral vanishes by symmetry
considerations, so in turn
\begin{equation*}\begin{split}
\int_{\Sph^1_K} &\left( \frac25(a_{11}^2-a_{12}) - \frac12 B_2^2
\right) \xi_1^2 \, d\sigma \\
&=\frac25(a_{11}^2-a_{12}) \int_{\Sph^1_K} \xi_1^2 \, d\sigma(\xi)
- \frac12 a_2^2 \int_{\Sph^1_K} \xi_1^2 \xi_2^2 \, d\sigma(\xi)
- \frac12 a_{11}^2 \int_{\Sph^1_K} \xi_1^6 \, d\sigma(\xi) \\
& = 2(a_{11}^2-a_{12}) \int_{\B^1_K} \eta_1^2 \, d\eta
- \frac92 a_2^2 \int_{\B^1_K} \eta_1^2 \eta_2^2 \, d\eta
- \frac92 a_{11}^2 \int_{\B^1_K} \eta_1^6 \, d\eta \\
& = 2(a_{11}^2-a_{12}) \left( \frac85 \frac1{\sqrt{2\pi}}
  \Gamma(\tfrac34)^2 \right)
- \frac92 a_2^2 \left( \frac{16}{45} \frac1{\sqrt{2\pi}}
  \Gamma(\tfrac34)^2 \right)
- \frac92 a_{11}^2 \left( \frac8{15} \frac1{\sqrt{2\pi}}
  \Gamma(\tfrac34)^2 \right) \\
& = \beta_1 \left( \varphi_{11}(0)^2 - 16 \varphi_{12}(0) -
  8 \varphi_2(0)^2 \right)
\end{split}\end{equation*}
as stated in the theorem, where we recall that $a_2 = \varphi_2(0)$,
$a_{12} = \varphi_{12}(0)$, and $a_{11} = \tfrac12 \varphi_{11}(0)$.

To complete the proof we need to remove the normalizing assumptions
\[
x=o \quad\text{and}\quad HT_x\Sigma=\spa\{e_2\}.
\]
To do this, we start with any
noncharacteristic point $x$ on a surface $\Sigma$,
use a preliminary isometry of $(\Heis,d_H)$ to move $x$ to the origin and
suitably rotate the horizontal tangent space $HT_{x}\Sigma$.
Precisely, once we have translated $\Sigma$ to move $x$ to the origin, we choose a suitable rotation about the $e_3$-axis
$$
R_\theta:\Heis \to \Heis \qquad \qquad R_\theta\begin{pmatrix} y_1 \\ y_2 \\ y_3  \end{pmatrix} = \begin{pmatrix} y_1\cos\theta-y_2\sin\theta \\ y_1\sin\theta + y_2\cos\theta \\ y_3  \end{pmatrix}.
$$
Since $d\ell_{x}^{-1} HT_{x}\Sigma$ is a horizontal line through the origin of $\Heis$, we may choose $\theta$ such that
\begin{equation}\label{X}
dR_\theta^{-1}d\ell_{x}^{-1} HT_{x}\Sigma = \text{\rm span}\left\{\begin{pmatrix} 0 \\ 1 \\ 0 \end{pmatrix}\right\}.
\end{equation}
It is then natural to define $\tilde\Sigma = R_\theta^{-1}\ell_x^{-1}\Sigma$
and observe that it is normalized as in the initial assumptions of the proof.
We have $o \in \tilde\Sigma$, $HT_0\tilde\Sigma = \spa e_2$
and by the implicit function theorem, locally near $o$ we may
find $\varphi$ such that
$\tilde\Sigma = \{ y \, : \, y_1 = \varphi(y_2,y_3) \}$.
We introduce the following local defining function
\[
v_{\tilde\Sigma}(y)=y_1-\varphi(y_2,y_3)
\]
to fit precisely the same assumption of $v_\Sigma$ at the beginning
of the proof, namely $X_1v_{\tilde\Sigma}(o)>0$. We will compute the quantity
$$
\varphi_{11}(0)^2 - 16 \varphi_{12}(0) - 8 \varphi_2(0)^2
$$
in terms of partial derivatives of a suitable defining function of $\Sigma$.
We express $\Sigma$ as level set of $u$ defined by the relation
\begin{equation}\label{eq:uv_Sigma}
u(\ell_x R_\theta(y)) = v_{\tilde\Sigma}(y).
\end{equation}
Equation \eqref{X} tells us that
$$
HT_{x}\Sigma =
\text{\rm span}\left\{-\sin\theta  X_1(x)+ \cos\theta X_2(x)\right\}=
 \text{\rm span}\left\{\begin{pmatrix} -\sin\theta \\ \cos\theta \\
 -2x_2 \sin\theta-2x_1\cos\theta  \end{pmatrix}\right\}.
$$
Differentiating \eqref{eq:uv_Sigma} we get
\[
0<X_1 v_{\tilde\Sigma}(o)=\cos\theta X_1u(x) + \sin\theta X_2u(x).
\]
We know that $-\sin\theta X_1u(x)+\cos\theta X_2u(x)=0$, therefore
the unit vector $ (X_1u(x),X_2u(x)) $ is a multiple of $(\cos\theta,\sin\theta)$
and the previous positivity condition yields
\[
(X_1u(x),X_2u(x)) = (\cos\theta,\sin\theta) =: (\overline{p},\overline{q}).
\]
According to the results of the previous section, the invariance of $\cS^3$ under rotations $R_\theta$ and left translations, it follows that
\begin{equation}\begin{split}
\cS^3(B_H(x,r)\cap\Sigma) &= \cS^3(B_H(o,r)\cap\tilde\Sigma) \\ &= \omega_Hr^3+c_o^1(\varphi_{11}(0)^2-16\varphi_{12}(0)-8\varphi_2(0)^2)r^5+O(r^6).
\end{split}\end{equation}
Note that
$$
\ell_x R_\theta(y) = \begin{pmatrix} x_1 + y_1\cos\theta-y_2\sin\theta \\ x_2 + y_1\sin\theta+y_2\cos\theta \\ x_3 + y_3 - 2(x_1y_1+x_2y_2)\sin\theta+2(x_2y_1-x_1y_2)\cos\theta \end{pmatrix}.
$$
We conclude that $\Sigma$ is defined, locally near $x$, by the implicit equation
\begin{equation}\begin{split}\label{implicit-equation}
&u\bigl(x_1 + \varphi(y_2,y_3)\cos\theta-y_2\sin\theta,x_2 + \varphi(y_2,y_3)\sin\theta+y_2\cos\theta, \bigr. \\ &\qquad \bigl. x_3 + y_3 - 2(x_1\varphi(y_2,y_3)+x_2y_2)\sin\theta+2(x_2\varphi(y_2,y_3)-x_1y_2)\cos\theta\bigr) = 0,
\end{split}\end{equation}
valid for $(y_2,y_3)$ in a neighborhood of $(0,0) \in \R^2$.
Let us introduce the auxiliary function
$\Upsilon(y_2,y_3,\theta)$
defined as the argument of $u$ in \eqref{implicit-equation}.

Repeated differentiation of \eqref{implicit-equation} yields formulas for iterated partial derivatives of $\varphi$ in terms of $u$. For instance, differentiating in \eqref{implicit-equation} once with respect to $y_2$ gives
\begin{equation}\begin{split}\label{implicit-equation-2}
&u_1(\Upsilon(y_2,y_3,\theta))(-\sin\theta+\varphi_1\cos\theta)+u_2(\Upsilon(y_2,y_3,\theta))(\cos\theta+\varphi_1\sin\theta) \\ &\qquad \qquad +u_3(\Upsilon(y_2,y_3,\theta)) (-2x_1\cos\theta-2x_2\sin\theta-2x_1\varphi_1\sin\theta+2y_1\varphi_1\cos\theta) = 0,
\end{split}\end{equation}
where $\varphi_1$ is evaluated at $(y_2,y_3)$ and subscripts denote partial derivatives. Setting $y_2=y_3=0$ gives
$$
u_1(x)(-\sin\theta)+u_2(x)(\cos\theta)+u_3(x)(-2x_1\cos\theta-2x_2\sin\theta) = 0,
$$
i.e.,
$$
X_2u(x)\cos\theta-X_1u(x)\sin\theta=0
$$
which was already known to be true by the definition of $\theta$. On the other hand, differentiating in \eqref{implicit-equation} once with respect to $y_3$ gives
\[
\begin{split}
&u_1(\Upsilon(y_2,y_3,\theta))(\varphi_2\cos\theta)+u_2(\Upsilon(y_2,y_3,\theta))(\varphi_2\sin\theta) \\
&\qquad\qquad+u_3(\Upsilon(y_2,y_3,\theta)) (1-2x_1\sin\theta+2x_2\cos\theta) = 0.
\end{split}
\]
Again setting $y_2=y_3=0$ and combining terms gives
$$
(X_1u(x)\cos\theta+X_2u(x)\sin\theta)\varphi_2(0)+u_3(x)=0.
$$
We can rewrite the previous equality as
$$
\varphi_2(0)=  - \frac{u_3(x)}{|\nabla_0u(x)|} = -\frac14 \cP_0(x),
$$
where $\cP_0$ denotes the imaginary curvature of Arcozzi--Ferrari. Our defining function $u$ is chosen to be compatible with the orientation of the surface $\Sigma$ in the sense of Arcozzi--Ferrari, indeed, in this case $\nabla_0 u(x)$ coincides with the outward unit horizontal normal to $\Sigma$ at $x$, locally relative to the open set $\{u<0\}$.

Differentiating once in \eqref{implicit-equation-2} with respect to $y_2$, evaluating at $y_2=y_3=0$ and combining terms yields
$$
\varphi_{11}(0) = -\frac{X_{11}u(X_2u)^2-(X_{12}u+X_{21}u)X_1uX_2u+X_{22}u(X_1u)^2}{|\nabla_0u|^3}(x) = -\cH_0(x),
$$
where $\cH_0$ denotes the horizontal mean curvature.

Finally, differentiating once in \eqref{implicit-equation-2} with respect to $y_3$, evaluating and $y_2=y_3=0$ and combining terms yields
\begin{equation*}\begin{split}
&-X_1X_3u\sin\theta+X_2X_3u\cos\theta \\ & \qquad +\varphi_2(0)[\tfrac12(X_1X_2u+X_2X_1u)(\cos^2\theta-\sin^2\theta)+(X_2X_2u-X_1X_1u)\cos\theta\sin\theta] \\ &\qquad \qquad + \varphi_{12}(0)(X_1u\cos\theta+X_2u\sin\theta) = 0
\end{split}\end{equation*}
where all iterated derivatives of $u$ with respect to the vector fields $X_1,X_2,X_3$ are evaluated at $x$. Inserting the known value of $\varphi_2(0)$ and solving for $\varphi_{12}(0)$ gives
\begin{equation*}\begin{split}
\varphi_{12}(0) &= -\frac1{|\nabla_0u|}\left( \frac{-X_1X_3uX_2u+X_2X_3uX_1u}{|\nabla_0u|} - \right. \\ & \qquad \qquad \left. \frac{X_3u[\tfrac12(X_1X_2u+X_2X_1u)((X_1u)^2-(X_2u)^2)+(X_2X_2u-X_1X_1u)X_1uX_2u]}{|\nabla_0u|^3} \right),
\end{split}\end{equation*}
or (after some algebra)
\begin{equation*}
\begin{split}
\varphi_{12}(0) &= \frac{X_2u}{|\nabla_0u|} X_1\left( \frac{X_3u}{|\nabla_0u|} \right) - \frac{X_1u}{|\nabla_0u|} X_2\left( \frac{X_3u}{|\nabla_0u|}\right) - 2 \left( \frac{X_3u}{|\nabla_0u|} \right)^2\\
&= - \tfrac14 \vec{e}_1(\cP_0) - \tfrac18 \cP_0^2.
\end{split}
\end{equation*}
Finally,
$$
\varphi_{11}(0)^2-16\varphi_{12}(0)-8\varphi_2(0)^2 = \cH_0^2 + \frac32 \cP_0^2 + 4 \vec{e}_1(\cP_0) \, .
$$
This completes the proof of Theorem \ref{theoC}.
\end{proof}

We conclude this section by giving the proof of Proposition \ref{prop:no-quadratic-t-graphs-as-3-uniform-supports}, which asserts that no quadratic $x_3$-graph can be the support of a $3$-uniform measure. Fixing a symmetric $2\times 2$ matrix $\bD$, we consider the quadratic surface $\Sigma$ defined by the equation $u(x) = 0$, where
\begin{equation}\label{eq:u}
u(z,x_3) = \langle z,\bD z\rangle - x_3.
\end{equation}
The following lemma provides an explicit expression for our partial differential operator $\cH_0^2 + \tfrac32 \cP_0^2 + 4 \vec{e}_1(\cP_0)$ on such a surface. We would again like to acknowledge Andrea Merlo for pointing out to us the expression in \eqref{eq:merlo}.

\begin{lemma}\label{lem:reconciliation}
Let $\Sigma$ be the quadratic $x_3$-graph defined by the equation $u(x)=0$, where $u(x)$ is as in \eqref{eq:u}. At a point $x=(z,x_3) \in \Sigma$, the expression
$$
(\cH_0^2 + \tfrac32 \cP_0^2 + 4 \vec{e}_1(\cP_0))(x)
$$
is proportional to the quantity
\begin{equation}\label{eq:merlo}
\cM_\bD(x) := \frac{\langle \bD \bJ n_\bD(z),\bJ n_\bD(z)\rangle^2 - 2 - 8 \langle \bD n_\bD(z),\bJ n_\bD(z) \rangle}{|(\bD-\bJ)z|^2}\,,
\end{equation}
where we have defined
$$
n_\bD(z) = \frac{(\bD-\bJ)(z)}{|(\bD-\bJ)(z)|}\in{\mathbb S}^1,\quad
\vec{n}_\bD=n_\bD^1(z) X_1(z,q_\bD(z)) +n^2_\bD(z) X_2(z,q_\bD(z))
$$
and $q_\bD(z)=\langle z,\bD z\rangle$ for all $z\in\R^2$.
\end{lemma}

\begin{proof}[Proof of Lemma \ref{lem:reconciliation}]
We first show that it is enough to consider the case when the defining matrix $\bD$ is diagonal. We then verify the conclusion in the diagonal case.

Let $\bD$ be a symmetric $2\times 2$ matrix and write $\bD = \bQ^T \bL \bQ$ with $\bL$ diagonal and $\bQ$ is a rotation. 
As a consequence, we have the equalities
\begin{equation}\label{eq:bQD}
\bQ\bD=\bL\bQ \quad\text{and} \quad \bQ\bJ=\bJ\bQ.
\end{equation}
We introduce a new defining function $\tilde u$ as follows 
\[
u(z,x_3)=\langle \bQ z,\bL \bQ z\rangle-x_3=\tilde u(\bQ z,x_3),
\]
where $\tilde u(w,x_3)=\langle w,\bL w\rangle-x_3$. 
As a consequence, we have that $u(\bQ^{-1}w,x_3)=\tilde u(w,x_3)$.
Defining $\widetilde\Sigma=\tilde u^{-1}(0)$, there holds
\begin{equation}\label{eq:rotSigma}
\widetilde\Sigma=\bQ_0 \Sigma
\end{equation}
with $\bQ_0(z,x_3)=(\bQ z,x_3)$.
By \eqref{eq:bQD} we get
\[
\bQ n_\Sigma(z) = n_{\widetilde\Sigma}(\bQ z).
\]
This formula can be seen also as a consequence of \eqref{eq:rotSigma}.
Moreover, we get
\begin{equation*}\begin{split}
\cM_\bD(z,x_3) 
&= \frac{\langle \bQ \bD \bJ n_\bD(z),\bQ\bJ n_\bD(z)\rangle^2 - 2 - 8 \langle \bQ\bD n_\bD(z),\bQ\bJ n_\bD(z) \rangle}{|\bQ(\bD-\bJ)z|^2} \\
&= \frac{\langle \bL \bJ \bQ n_\bD(z),\bJ\bQ n_\bD(z)\rangle^2 -2-8 \langle \bL\bQ  n_\bD(z),\bJ\bQ n_\bD(z) \rangle}{|(\bL-\bJ)\bQ z|^2} \\
&= \frac{\langle \bL \bJ  n_\bL(\bQ z),\bJ n_\bL(\bQ z)\rangle^2 -2-8 \langle \bL n_\bL(\bQ z),\bJ n_\bL(\bQ z) \rangle}{|(\bL-\bJ)\bQ z|^2} \\
&= \cM_\bL(\bQ z,x_3)\,.
\end{split}\end{equation*}
The geometric quantities $\cH_0$ and $\cP_0$, 
being metrically defined, are invariant under the rotations $(z,x_3) \mapsto (\bQ z,x_3)$, which in this case are isometries of both Carnot--Carath\'eodory and Kor\'anyi distances. Precisely $\cP_0(z,t)=\widetilde\cP_0(\bQ z,t)$
and $\cH_0(z,t)=\widetilde\cH_0(\bQ z,t)$,
where we have denoted by $\widetilde\cP_0$ and $\widetilde\cH_0$ 
the imaginary curvature and the horizontal
mean curvature and imaginary curvature, respectively, 
of the rotated surface $\widetilde\Sigma$. 
It is not difficult to check by direct computation that
\[
\vec{e}_1\cP_0(z,x_3)=\langle \nabla \cP_0(z,x_3),\cJ \vec{n}_\Sigma(z,x_3)\rangle=\langle \nabla \widetilde\cP_0(\bQ z,x_3),\cJ \vec{n}_{\widetilde\Sigma}(\bQ z,x_3)\rangle=\vec{\tilde e}_1\widetilde\cP_0(\bQ z,x_3).
\]
As a consequence, we have proved that
$$
(\cH_0^2 + \tfrac32 \cP_0^2 + 4 \vec{e}_1(\cP_0))(z,x_3) = (\widetilde\cH_0^2 + \tfrac32 \widetilde\cP_0^2 + 4 \vec{\tilde e}_1(\widetilde\cP_0))(\bQ z,x_3).
$$
The preceding argument shows that in order to prove the lemma, it suffices to assume that $\bD$ is a diagonal matrix.

We therefore consider a surface $\Sigma$ defined by the equation $u(x)=0$, where $u(x) = \langle z,\bL z \rangle - x_3 = \lambda_1 x_1^2 + \lambda_2 x_2^2 - x_3$ and $x=(z,x_3)=(x_1,x_2,x_3)$. We compute
$$
X_1u = 2\lambda_1x_1 - 2x_2, \quad X_2u = 2\lambda_2x_2 + 2x_1,
$$
and
$$
|\nabla_0 u| = 2\sqrt{(\lambda_1x_1-x_2)^2+(\lambda_2x_2+x_1)^2} = 2|(\bL-\bJ)z|\,.
$$
Then
$$
n_\bL(z) = \left( \frac{\lambda_1x_1-x_2}{|(\bL-\bJ)z|},\frac{\lambda_2x_2+x_1}{|(\bL-\bJ)z|} \right)
$$
and
\begin{equation*}\begin{split}
\cH_0 &= X_1\left(\frac{X_1u}{|\nabla_0u|}\right) + X_2\left(\frac{X_2u}{|\nabla_0u|}\right) \\
&= \frac{\lambda_1(\lambda_2x_2+x_1)^2+\lambda_2(\lambda_1x_1-x_2)^2}{|(\bL-\bJ)z|^3} = \frac{\langle \bL\bJ n_\bL,\bJ n_\bL\rangle}{|(\bL-\bJ)z|}\,,
\end{split}\end{equation*}
$$
\cP_0 = \frac{4X_3u}{|\nabla_0u|} = \frac{-2}{|(\bL-\bJ)z|}\,,
$$
and
\begin{equation*}\begin{split}
\vec{e}_1(\cP_0)
&= \frac{X_1uX_2\cP_0-X_2uX_1\cP_0}{|\nabla_0u|} \\
&= \frac{\lambda_1x_1-x_2}{|(D-J)z|}X_2(\frac{-2}{|(\bL-\bJ)z|}) - \frac{\lambda_2x_2+x_1}{|(\bL-\bJ)z|}X_1(\frac{-2}{|(\bL-\bJ)z|}) \\
&= \frac{-2}{|(\bL-\bJ)z|^2} + \frac{2(\lambda_2-\lambda_1)(\lambda_1x_1-x_2)(\lambda_2x_2+x_1)}{|(\bL-\bJ)z|^4} \\
&= \frac{-2}{|(\bL-\bJ)z|^2} - \frac{2\langle \bL n_\bL,\bJ n_\bL\rangle}{|(\bL-\bJ)z|^2}.
\end{split}\end{equation*}
Thus
$$
\cH_0^2+\frac32\cP_0^2+4\vec{e}_1(\cP_0) = \frac{\langle \bL \bJ n_\bL,\bJ n_\bL\rangle^2 + 6 - 8 - 8 \langle \bL n_\bL,\bJ n_\bL\rangle}{|(\bL-\bJ)z|^2}
$$
which completes the proof of the lemma.
\end{proof}

\begin{proof}[Proof of Proposition \ref{prop:no-quadratic-t-graphs-as-3-uniform-supports}]
Let $\Sigma$ be a quadratic $x_3$-graph. As previously explained, we may assume without loss of generality that $\Sigma$ is defined by the equation $u(x)=0$, where $u(z,x_3) = \langle z,\bL z\rangle -x_3$ and
$$
\bL = \begin{pmatrix} \lambda_1 & 0 \\ 0 & \lambda_2 \end{pmatrix}
$$
is diagonal. The case $\lambda_1=\lambda_2$ has already been treated in Example \ref{ex:paraboloids}, so we assume $\lambda_1 \ne \lambda_2$. We will show that in this case, the expression
$$
\langle \bL \bJ n_\bL,\bJ n_\bL\rangle^2 - 2 - 8 \langle \bL n_\bL,\bJ n_\bL\rangle
$$
cannot vanish in any open subset of $\Sigma$. Clearing denominators, we consider the equation
$$
\langle \bL\bJ(\bL-\bJ)z,\bJ(\bL-\bJ)z\rangle^2 - 2|(\bL-\bJ)z|^4 - 8|(\bL-\bJ)z|^2\langle \bL(\bL-\bJ)z,\bJ(\bL-\bJ)z \rangle = 0.
$$
By direct computation, taking into account that
\[
|(\bL-\bJ)z|^2=(\lambda_1x_1-x_2)^2+(\lambda_2x_2+x_1)^2
\]
and writing the previous expression in terms of $A=\lambda_2x_2+x_1$ 
and $B=\lambda_1x_1-x_2$ we get
\begin{equation}\label{eq:expressAB}
(\lambda_1^2-2) A^4+(\lambda_2^2-2) B^4+
(2\lambda_1\lambda_2-4)A^2B^2+8(\lambda_2-\lambda_1) AB (A^2+B^2)=0.
\end{equation}
If the matrix $\bL-\bJ$ is invertible, namely $\lambda_1\lambda_2\neq-1$, the variable $(A,B)$ takes any vector of $\R^2$, hence
all coefficients of \eqref{eq:expressAB} are null, leading to the contradiction
$\lambda_1=\lambda_2$.
If $\lambda_1\lambda_2=-1$, then $B=\lambda_1A$ and the previous equation becomes
\[
\big\{\lambda_1^2-2+\lambda_1^4(\lambda_2^2-2)-6\lambda_1^2+8(\lambda_2-\lambda_1)\lambda_1(1+\lambda_1^2)\big\}A^4=0.
\] 
The coefficient of $A^4$ must vanish and can be written as
\[
-(10)(1+2\lambda_1^2+\lambda_1^4)
\]
that is never vanishing, hence getting a contradiction.
\end{proof}

\section{Integrals of monomials over the Kor\'anyi sphere and Kor\'anyi ball}\label{sec:koranyi-integrals}

In this section we prove Lemma \ref{polar-coordinates} on a polar coordinate decomposition for the Kor\'anyi-type norm $\nu_K$ in $\R^2$. We begin by stating some corollaries of this lemma. We consider weighted polynomials in the coordinates $\eta_1$ and $\eta_2$, where the weighting is defined by the dilations~$\delta_r$.

\begin{corollary}\label{polar-coordinates-corollary}
Let $h=h(\eta)$ be a monomial of weighted degree $d$. Then
$$
\int_{\Sph^1_K} h(\xi) \, d\sigma(\xi) = (d+3) \int_{\B^1_K} h(\eta)
\, d\eta.
$$
\end{corollary}

\begin{proof}
Compute
\begin{equation*}\begin{split}
\int_{\B^1_K} h(\eta) \, d\eta
= \int_{\Sph^1_K} \int_0^1 h(\delta_t\xi)
\,t^2\,dt\,d\sigma(\xi)
= \int_{\Sph^1_K} \int_0^1 h(\xi) \, t^{d+2} \, dt\,d\sigma(\xi).
\end{split}\end{equation*}
The conclusion follows upon integrating in $t$.
\end{proof}

In particular, we compute monomial integrals $\int_{\B_K^1}
\eta_1^\alpha\eta_2^\beta \, d\eta$ over the ball $\B_K^1$. If either
$\alpha$ or $\beta$ is odd, the integral is clearly zero.

\begin{lemma}\label{monomial-integrals}
For nonnegative integers $a,b$,
$$
\int_{\B^1_K} \eta_1^{2a}\eta_2^{2b} \, d\eta = \frac1{2b+1} \int_0^1
u^{a/2-3/4}(1-u)^{b+1/2} \, du =
\frac{\Gamma(\tfrac{a}2+\tfrac14)\Gamma(b+\tfrac12)}{2\Gamma(b+\tfrac{a}2+\tfrac74)}.
$$
\end{lemma}

\begin{proof}
Integrating first in $\eta_2$ yields
$$
\int_{\B^1_K} \eta_1^{2a}\eta_2^{2b} \, d\eta
= 4 \int_0^1 \eta_1^{2a} \int_0^{\sqrt{1-\eta_1^4}} \eta_2^{2b} \,
d\eta_2 \, d\eta_1
= \frac4{2b+1} \int_0^1 \eta_1^{2a} (1-\eta_1^4)^{b+1/2} \, d\eta_1.
$$
The change of variables $u=\eta_1^4$ gives the first conclusion of the
lemma. To obtain the second conclusion we use the standard integral
representation of the beta function:
$$
B(s,t) = \int_0^1 u^{s-1}(1-u)^{t-1} =
\frac{\Gamma(s)\Gamma(t)}{\Gamma(s+t)}
$$
valid for $s,t>0$.
\end{proof}

For instance, the value of $\omega_H := \Vol(\B^1_K)$ is obtained
from Lemma \ref{monomial-integrals} with $a=b=0$. We find
\begin{equation}\label{eq:omegaH}
\omega_H = \frac{\Gamma(\tfrac14)\Gamma(\tfrac12)}{2\Gamma(\tfrac74)}
= \frac13 \sqrt{\frac2\pi} \Gamma(\tfrac14)^2 \approx 3.49608
\end{equation}
via the reflection formula $\Gamma(\tfrac14)\Gamma(\tfrac34)=\pi\sqrt2$. 

\begin{proof}[Proof of Lemma \ref{polar-coordinates}]
The spheres $\Sph_K^1(r):=\{\eta:\nu_K(\eta)=r\}$, $r>0$, are smooth
submanifolds of $\R^2$. According to the classical coarea formula, for
an arbitrary integrable $h$ on $\B_K^1$ we have
\begin{equation}\label{coarea}
\int_{\B_K^1} h(\eta)\,d\eta = \int_0^1 \int_{\nu_K^{-1}(t)}
\frac{h(\eta)}{|\nabla\nu_K(\eta)|} \, d\cH^1(\eta)\,dt,
\end{equation}
where $\nabla$ denotes the standard Euclidean gradient. We convert the
inner integral into an integral over the unit sphere $\Sph_K^1$ via the
change of variables $\delta_t:\Sph_K^1\to\Sph_K^1(t)=\nu_K^{-1}(t)$.
For an integrable function $g:\Sph_K^1(t) \to \R$ we have
\begin{equation}\label{change-of-variables}
\int_{\Sph_K^1(t)} g(\eta) \, d\cH^1(\eta) = \int_{\Sph_K^1}
g(\delta_t\eta) \, d((\delta_{1/t})_\#\cH^1)(\eta).
\end{equation}
The measure $(\delta_{1/t})_\#\cH^1 \restrict \Sph_K^1$ is absolutely
continuous with respect to $\cH^1 \restrict \Sph_K^1$.
Recall that $(\delta_{1/t})_\#\cH^1(B) = \cH^1(\delta_t(B))$ for $B
\subset \Sph_K^1$. The Radon--Nikodym derivative of
$(\delta_{1/t})_\#\cH^1 \restrict \Sph_K^1$ with respect to $\cH^1
\restrict \Sph_K^1$ is given by
$$
\frac{d(\delta_{1/t})_\#\cH^1}{d\cH^1}(\eta)
=
\frac{|(\delta_t)_*(\vec{T})|}{|\vec{T}|}
$$
for any nonzero tangent vector $\vec{T} \in T_\eta\Sph_K^1$, where
$|\cdot|$ denotes the usual Euclidean norm. We have
$\nabla\nu_K=\nu_K^{-3}(\eta_1^3,\tfrac12\eta_2)$ and so $\vec{T} =
\nu_K^{-3}(\tfrac12\eta_2,-\eta_1^3)$ is a nonzero tangent vector at
$\eta$. Evaluating the above Radon--Nikodym derivative, we find
$$
\frac{d(\delta_{1/t})_\#\cH^1}{d\cH^1}(\eta) = t
\left( \frac{\tfrac14\eta_2^2+t^2\eta_1^6}{\tfrac14\eta_2^2+\eta_1^6}
\right)^{1/2} \,.
$$
We also have
$$
|\nabla\nu_K(\eta)| = \frac{\sqrt{\eta_1^6 + \tfrac14
    \eta_2^2}}{\nu_K^3(\eta)}
$$
and so, for $\eta\in\Sph_K^1$ and $t>0$,
$$
|\nabla\nu_K(\delta_t\eta)| = \frac{\sqrt{t^6\eta_1^6 + \tfrac14
    t^4\eta_2^2}}{t^3} = \frac{\sqrt{t^2\eta_1^6+\tfrac14\eta_2^2}}{t}\,.
$$
Returning to \eqref{coarea} and using \eqref{change-of-variables} with
$g=h/|\nabla\nu_K|$ we obtain
\begin{equation*}\begin{split}
\int_{\B_K^1} h(\eta)\,d\eta
&= \int_0^1 \int_{\Sph_K^1}
\frac{t\,h(\delta_t\eta)}{\sqrt{t^2\eta_1^6+\tfrac14\eta_2^2}} \,
t\left(\frac{\tfrac14\eta_2^2+t^2\eta_1^6}{\tfrac14\eta_2^2+\eta_1^6}\right)^{1/2}\,d\cH^1(\eta)\,dt
\\
&= \int_0^1 \int_{\Sph_K^1}
\frac{t^2\,h(\delta_t\eta)}{\sqrt{\eta_1^6+\tfrac14\eta_2^2}} \,d\cH^1(\eta)\,dt\,.
\end{split}\end{equation*}
The conclusion holds with the Radon measure $d\sigma(\eta) =
(\eta_1^6+\tfrac14\eta_2^2)^{-1/2}\,d\cH^1(\eta)$.
\end{proof}

\bibliographystyle{acm}
\bibliography{kp}
\end{document}